\documentclass[]{article}

\title{Representation of harmonic functions with respect to subordinate Brownian motion}

\author{Ivan Bio\v{c}i\'{c}}
\date{}

\usepackage[utf8]{inputenc}

\usepackage{amsmath}
\usepackage{amsthm}
\usepackage{amsfonts}
\usepackage{mathtools}
\usepackage{bbm} 
\usepackage[dvipsnames]{xcolor} 
\usepackage{amssymb} 
\usepackage{enumerate}
\usepackage{comment}

\usepackage{hyperref}
\hypersetup{
	colorlinks=true,
	linkcolor=green,
}
\DeclareMathOperator\supp{supp}

\newtheorem{thm}{Theorem}[section]
\newtheorem{prop}[thm]{Proposition}
\newtheorem{cor}[thm]{Corollary}
\newtheorem{defn}[thm]{Definition}
\newtheorem{lem}[thm]{Lemma}
\theoremstyle{definition}
\newtheorem{rem}[thm]{Remark}

\newtheorem*{assumption*}{\assumptionnumber}
\providecommand{\assumptionnumber}{}
\makeatletter
\newenvironment{assumption}[2]
{%
	\renewcommand{\assumptionnumber}{(\textbf{#1#2})}%
	\begin{assumption*}%
		\protected@edef\@currentlabel{(\textbf{#1#2})}%
	}
	{%
	\end{assumption*}
}
\makeatother

\usepackage{xpatch}
\makeatletter
\AtBeginDocument{\xpatchcmd{\@thm}{\thm@headpunct{.}}{\thm@headpunct{}}{}{}}
\makeatother

\newcommand{\R}{\mathbb{R}}
\newcommand{\N}{\mathbb{N}}
\newcommand{\p}{\mathbb{P}}
\newcommand{\ex}{\mathbb{E}}
\newcommand{\subsub}{\subset\subset}

\newcommand{\BB}{\mathcal{B}}

\newcommand{\setword}[2]{%
	\phantomsection
	#1\def\@currentlabel{\unexpanded{#1}}\phantomsection\label{#2}%
}

\begin{document}
	\maketitle

	\begin{abstract}
		In this article we prove a representation formula for non-negative generalized harmonic functions with respect to a subordinate Brownian motion in a general open set $D\subset \R^d$. We also study oscillation properties of quotients of Poisson integrals and prove that oscillation can be uniformly tamed.
	\end{abstract}
	

	
	\section{Introduction}
	The goal of this article is to prove a representation formula for non-negative harmonic functions with respect to a class of subordinate Brownian motions in a general open set $D\subset \R^d$, $d\ge 2$, where the Laplace exponent of the corresponding subordinator is a complete Bernstein function satisfying certain weak scaling conditions. In this setting, the novelty is  that we look at pairs $(f,\lambda)$ such that $f$ is a function on $D$ and $\lambda$ is a measure on $D^c$ that we call, following \cite{bogdan_est_and_struct}, functions with outer charge. We prove the following result: if $f$ is a non-negative harmonic function in $D$ with  a  non-negative outer charge $\lambda$, then there is a unique finite measure $\mu$ on $\partial D$ such that
	\begin{align}\phantomsection\label{eq:repr-pocetak}
		f=P_D\lambda+M_D\mu,\quad \text{in } D.
	\end{align}
	Here $P_D\lambda$ denotes the Poisson integral of the measure $\lambda$ and $M_D\mu$ the Martin integral of the measure $\mu$, see Theorem \ref{t:reprezentacija_L_harm}. Such representation  was proved for the case of the isotropic $\alpha$-stable process in \cite{bogdan_est_and_struct} more than 10 years ago. A similar representation for   functions (in the classical sense)  was proved recently  for more general Markov processes   in bounded open sets in \cite{juscisin_kwasina_martin_kernels}, and in {\it nice} and general open sets in \cite{vondra_RMI_2018}. Analogous result  for non-negative classical harmonic functions on the ball $B(x,r)$, i.e. harmonic functions with respect to the Brownian motion,  is better known as Riesz-Herglotz theorem, cf. \cite{armitage2012classical}. In the article the case $d=1$ is excluded since it would require somewhat different potential theoretic methods.
	
	On the way to obtaining the representation, motivated by results in \cite{bogdan_est_and_struct}, we study  the  relative oscillation of  the  quotient of Poisson integrals. The novelty of this results is that we prove that the oscillation can be uniformly tamed. To be more precise, for a positive function $f$ on a set $D$ we define the relative oscillation of the function $f$ by 
	$$\mathrm{RO}_Df\coloneqq \frac{\sup_D f}{\inf_D f}.$$
	We prove that for every $\eta>0$ there is $\delta>0$ such that for every $D\subset B(0,R)$ and measures $\lambda_1$ and $\lambda_2$ on $B(0,R)^c$ we have
	$$ \mathrm{RO}_{D\cap B(0, \delta)}\frac{P_D\lambda_1}{P_D\lambda_2}\le 1+\eta,$$ see Lemma \ref{lemma 8}. Uniformity lies in the fact  that $\delta$ is independent of the set $D$ and the measures $\lambda_1$ and $\lambda_2$. Similar claims on the relative oscillation of harmonic functions were recently proved for more general processes in \cite[Proposition 2.5 \& Proposition 2.11]{vondra_RMI_2018} and \cite[Theorem 2.4 \& Theorem 2.8]{vondra_fm_2016} but the claims lack the aforementioned uniformity.
	
	In the article we also study the  boundary trace  operator $W_D$, see Definition \ref{def_1.13}. The operator $W_D$ was introduced in \cite{bogdan_et_al_19} building on results in \cite{bogdan_est_and_struct}. In \cite{bogdan_et_al_19} it plays a significant role in the semilinear Dirichlet problem for the fractional Laplacian. We generalize the operator for the case of the subordinate Brownian motion and use it as a tool to obtain  the  finite measure for the Martin integral in the representation.

	Motivated by the article \cite{bogdan_est_and_struct} where harmonic functions with outer charge were introduced for the case of the isotropic $\alpha$-stable process, we use the same concept to define $L$-harmonic functions with outer charge, see Definition \ref{d:L-harmonic function}. The letter $L$ stands for the integrodifferential operator $L$ which generates the subordinate Brownian motion, see \eqref{operator L}. In Theorem \ref{t:harmon L=0} we prove that $L$ annihilates all $L$-harmonic functions in the weak sense. Also, the novelty of the study of $L$-harmonic functions is that we prove that all such functions are continuous, see Proposition \ref{p:neprekidnost harmonijskih}, whereas in \cite{bogdan_est_and_struct} the continuity condition was used as a part of the definition.  Moreover, motivated by results in \cite{grzywny_potential_kernels}, in Theorem \ref{r:neprekidnost harmoncijskih Cinfty} we prove even stronger result which says that every $L$-harmonic function is infinitely differentiable.

	The article is organized as follows. Below this paragraph we introduce the notation. In Section \ref{s:Preliminaries} we define the process of interest, introduce the Green and  the  Poisson kernels, and state some well-known results on the process that will be needed in the article. In Section \ref{s:Poisson and harmonic} we prove basic results on the Poisson kernel, define $L$-harmonic functions and study their basic properties. In Section \ref{s:Martin kernel} we recall already known facts on the theory of the Martin kernel and connect them to $L$-harmonic functions. Section \ref{s:W_D and representation} begins with results on the boundary trace operator $W_D$. After we prove results on  the  relative oscillations of  the  Poisson integrals, we finish the article by proving the representation formula for non-negative $L$-harmonic functions.

	\paragraph{Notation.}	
	For an open set $D\subset \R^d$: $C(D)$ denotes the set of all continuous functions on $D$, $C^2(D)$ twice continuously differentiable functions on $D$, $C^\infty(D)$ infinitely differentiable functions on $D$, and $C_c^\infty(D)$ infinitely differentiable functions with  compact support on $D$. Furthermore, $L^1(D)$ is the set of all integrable functions on $D$, and $L^1_{loc}(D)$ the set of all locally integrable functions on $D$, with respect to the Lebesgue measure restricted on $D$. If $D=\R^d$ we write $L^1$ and $L^1_{loc}$ instead of $L^1(\R^d)$ and $L^1_{loc}(\R^d)$, respectively. The boundary of the set $D$ is denoted by $\partial D$. Notation $U\subsub D$ means that $U$ is a  nonempty  bounded open set such that $U\subset \overline U\subset D$ where $\overline U$ denotes  the  closure of $U$. By $|x|$ we denote the Euclidean norm of $x\in\R^d$ and $B(x,r)$ denotes  the  ball around $x\in\R^d$ with radius  $r>0$. We abbreviate $B_r:=B(0,r)$. For $A,B\subset \R^d$ let $\delta_A(x)=\inf\{|x-y|:y\in A^c\}$ and $\mathrm{dist}(A,B)=\inf\{|x-y|:x\in A,y\in B\}$. Unimportant constants in the article will be denoted by small letters $c$, $c_1$, $c_2$, $\dots$, and their labeling starts anew in each new statement. By a big letter $C$ we denote some more important constants, where e.g. $C(a,b)$ means that the constant $C$ depends only on parameters $a$ and $b$.  However, the dependence on the dimension $d$ will not be mentioned explicitly. All constants are positive finite numbers.  Furthermore, in what follows  when we say $\nu$ is a measure,  we mean that $\nu$ is a non-negative measure on $\R^d$. By $|\nu|$ we denote the total variation of a signed measure $\nu$. When we say $\nu$ is a signed measure on $D\subset \R^d$, we mean that $\nu$ is a signed measure on $\R^d$ and $|\nu|(D^c)=0$. The  Dirac measure of a point $x\in\R^d$ is denoted by $\delta_x$. Finally, $\BB(\R^d)$ denotes  Borel  measurable sets in $\R^d$,  and we suppose that all functions in the article are Borel functions  and all signed measures are Borel signed measures.

	\section{Preliminaries}\phantomsection\label{s:Preliminaries}
	\subsection{Process and the jumping kernel}
	
	Let $S = (S_t)_{t\ge 0}$ be a subordinator with the Laplace exponent $\phi$, i.e. $S$ is an increasing L\'evy process with $S_0 = 0$ and
	\begin{align*}
		\mathbb{E}[e^{-\lambda S_t}]=e^{-t\phi(\lambda)},\quad \lambda,t\ge 0.
	\end{align*}
	It is well known that $\phi$ is a Bernstein function of the form
	\begin{align}\phantomsection\label{e:laplace exp}
		\phi(\lambda)=b\,\lambda+\int_0^\infty(1-e^{-\lambda t})\mu(dt),\quad \lambda>0,
	\end{align}
	where $b\ge0$ and $\mu$ is a measure on $(0,\infty)$ satisfying $\int_0^\infty (1 \wedge t)\mu(dt) < \infty$. The measure $\mu$ is called the L\'evy measure and $b$ the drift of the subordinator. Throughout this article we suppose that $\phi$ is a complete Bernstein function. This assumption means that $\mu(dt)$ has a density $\mu(t)$ which is a completely monotone function. For details about Bernstein functions see \cite{bernstein}. Also, we suppose that $\phi$ satisfies the following \textit{upper and lower scaling conditions at infinity}:
	\begin{assumption}{H}{1}\phantomsection\label{H1} There exist constants $\delta_1$, $\delta_2$ $\in(0, 1)$ and $a_1, a_2 > 0$  such that
		\begin{align}
			\phi(\lambda r)&\ge a_1\lambda^{\delta_1}\phi(r),\quad \lambda \ge 1,\, r\ge1, \tag{\text{LSC}}\phantomsection\label{LSC}\\
			\phi(\lambda r)&\le a_2\lambda^{\delta_2}\phi(r),\quad \lambda \ge 1,\, r\ge 1. \tag{\text{USC}}\phantomsection\label{USC}
		\end{align}
	\end{assumption}
	\noindent This assumption yields that $b=0$. 
	
	 Suppose that $W=(W_t)_{t\ge 0}$ is a Brownian motion in $\R^d$, $d\ge 2$,  independent of $S$  with the characteristic exponent $\xi\mapsto|\xi|^2$, $\xi\in\R^d$. 
	The process $X=((X_t)_{t\ge 0},(\mathbb{P}_x)_{x\in\R^d})$ defined as $X_t=W_{S_t}$ is called a subordinate Brownian motion in $\R^d$. Here $\mathbb{P}_x$ denotes the probability under which the process $X$ starts from $x\in\R^d$, and  by  $\mathbb{E}_x$  we denote  the corresponding expectation. Under conditions above $X$ is a pure-jump rotationally symmetric L\'evy process with the characteristic exponent $\xi\mapsto\Psi(\xi)=\phi(|\xi|^2)$. The exponent has the following form
	\begin{align*}
		\Psi(\xi)=\phi(|\xi|^2)=\int_{\R^d}\left(1-\cos(\xi\cdot x)\right) J(dx),\quad \xi\in\R^d,
	\end{align*}
	where the measure $J$ satisfies $\int_{\R^d} (1\wedge |x|^2)J(dx)<\infty$ and it is called the L\'evy measure of the process $X$. Also, $J$ has a density given by $J(x)=j(|x|)$, $x\in\R^d$, where
	\begin{align*}
		j(r)\coloneqq \int_0^\infty (4\pi t)^{-d/2} e^{-r^2/(4t)}\mu(t)dt,\quad r>0,
	\end{align*}
	The density $j$ is positive, continuous, decreasing and satisfies $\lim\limits_{r\to\infty}j(r)=0.$

	It is well known that  since   $\phi$ is a complete Bernstein function, there is a constant $C=C(\phi)>0$ such that
	\begin{align}\phantomsection\label{eq:scm_2.12}
		j(r)\le C j(r+1),\quad r\ge1,
	\end{align}
	see e.g. \cite[Eq. (2.12)]{vondra_2012_SCM}. Also, from \cite[Lemma 4.3]{vondra_fm_2016} we have that for every $r_0\in(0,1)$
	\begin{align}\phantomsection\label{eq:derivacija j}
		\lim_{\delta\to0}\sup_{r>r_0}\frac{j(r)}{j(r+\delta)}=1.
	\end{align}
	Using \eqref{eq:derivacija j} we can easily prove the following technical lemma.
	\begin{lem}\phantomsection\label{eq:comparability of j}
		Let $R>0$, $\varepsilon>0$, and $0<q\le1$. There exists $p=p(q,\varepsilon, R)<q$ such that for all $z\in B_{pR}$ and $y\in B_{qR}^c$
		\begin{align*}
			\frac{1}{1+\varepsilon}j(|y|)\le j(|y-z|)\le (1+\varepsilon)j(|y|).
		\end{align*}
	\end{lem}
	
	\begin{rem}
		Condition $r\ge 1$ in \eqref{LSC} and \eqref{USC} is important in the sense that  the  scaling is true away from zero. Using the continuity of $\phi$ it is easy to show that if $R_0>0$, then \eqref{LSC} and \eqref{USC} are also valid for $r\ge R_0$ but with different constants $a_1$ and $a_2$ ($\delta_1$ and $\delta_2$ remain the same).
		Similarly, since $j$ is continuous,  inequality  \eqref{eq:scm_2.12}
		holds for $r\ge R_0$ with a different constant $C$.
	\end{rem}
	
	\subsection{Additional assumptions}
	In some results dealing with unbounded sets we will  occasionally make additional assumptions  on the density $j$ and the exponent $\phi$. The first assumption strengthens \ref{H1}.
	\begin{assumption}{H}{2}(Global scaling condition)\phantomsection\label{H2}
		There exist constants $\delta_1$, $\delta_2$ $\in(0, 1)$ and $a_1, a_2 > 0$  such that 
		\begin{align}
			\phi(\lambda r)&\ge a_1\lambda^{\delta_1}\phi(r),\quad \lambda \ge 1,\, r>0, \tag{\text{GLSC}}\phantomsection\label{GLSC}\\
			\phi(\lambda r)&\le a_2\lambda^{\delta_2}\phi(r),\quad \lambda \ge 1,\, r>0. \tag{\text{GUSC}}\phantomsection\label{GUSC}
		\end{align}
	\end{assumption}

	The second  assumption comes as an addition to Lemma \ref{eq:comparability of j}.
	\begin{assumption}{E}{}\phantomsection\label{E2}
		For every $R\ge 1$, $\varepsilon>0$, and $q\in(1,\infty)$, there exists $p=p(q,\varepsilon, R)>q$ such that for all $z\in B_{pR}^c$ and $y\in B_{qR}$
		\begin{align*}
			\frac{1}{1+\varepsilon}j(|z|)\le j(|y-z|)\le (1+\varepsilon)j(|z|).
		\end{align*}
	\end{assumption}
	To the best of our knowledge it is not clear if the assumption \ref{E2} is true for every density $j$ generated by a complete Bernstein function. However, it is known that if for some $\alpha\in(0,2)$ we have $\lim_{\lambda\to 0}\frac{\phi(\lambda^2)}{\lambda^\alpha l(\lambda)}=1$, where $l$ is a slowly varying function at 0, then the condition \ref{E2} is satisfied, see \cite[Section 4.2]{vondra_fm_2016}.
	
	Note that the isotropic $\alpha$-stable process, $\alpha\in(0,2)$, satisfies all mentioned assumptions, since in this case we have $\phi(\lambda)=\lambda^{\alpha/2}$ and $j(r)=c(d,\alpha)\frac{1}{r^{d+\alpha}}$.

	\subsection{Operator $L$}\phantomsection\label{ss:harm functions and operators}
	
	For $x\in\R^d$ and $u:\R^d\to\R$ we let
	\begin{align}\phantomsection\label{operator L}
		Lu(x)&\coloneqq \mathrm{P.V.}\int\limits_{\R^d}(u(y)-u(x))j(|y-x|)dy\\
		&\coloneqq\lim\limits_{\varepsilon\to 0^+}\int_{|y-x| >  \varepsilon}(u(y)-u(x))j(|y-x|)dy,\nonumber
	\end{align}
	whenever the limit above exists. In the case of the isotropic $\alpha$-stable process the operator $L$ is the fractional Laplacian $\Delta^{\alpha/2}$.
	
	If $\varphi\in C_c^2(\R^d)$, i.e. $\varphi$ is a twice continuously differentiable function with compact support, then $L\varphi(x)$ exists for every $x\in \R^d$. In fact, if $\varphi\in C_c^2(\R^d)$, then using Taylor's theorem it is easy to see  that there is a constant $C= C(K,\phi)>0$,  where $\supp \varphi\subset K\subsub\R^d$, such that
	\begin{align}\phantomsection\label{e: L C_c}
		|L\varphi(x)|\le C ||\varphi||_{C^2(\R^d)} (1\wedge j(|x|)),\quad x\in\R^d.
	\end{align}
	Here $||\cdot||_{C^2(\R^d)}$ denotes the standard norm for twice differentiable functions.
	
	For functions $u\in \mathcal{L}^1\coloneqq L^1(\R^d,(1\wedge j(|x|))dx)$ we define  the  distribution $\tilde{L}$ as
	\begin{align*}
		\langle \tilde{L}u,\varphi\rangle\coloneqq \langle u,L\varphi\rangle\coloneqq \int_{\R^d}u(x)L\varphi(x)dx,\quad \varphi\in C_c^\infty(\R^d).
	\end{align*}
	 The condition $u\in \mathcal{L}^1$ is needed to ensure that the integral above is well defined, see \eqref{e: L C_c}. Also, note that since $j$ is positive, we have $\mathcal{L}^1\subset L^1_{loc}$. Following \cite[Section 3]{bogdan1999potential},  it is easy to show that if $u\in C^2(D)\cap \mathcal{ L}^1$, then  $Lu(x)$  exists for every $x\in D$ and $\tilde{L}u=Lu$ as distributions on $D$, i.e.
	\begin{align*}
		\langle \tilde{L}u,\varphi\rangle=\langle Lu,\varphi\rangle,\quad \varphi\in C_c^\infty(D).
	\end{align*}
	Furthermore, we extend the definition of $\tilde{L}$  to  measures in the following way
	\begin{align}\phantomsection\label{eq:extension of L}
		\langle \tilde{L}\lambda,\varphi\rangle\coloneqq\langle \lambda,L\varphi\rangle\coloneqq \int_{\R^d}L\varphi(x)\lambda(dx),
	\end{align} 
	for all signed measures $\lambda$  such that $\int_{\R^d}(1\wedge j(|x|))|\lambda|(dx)<\infty$.

	\subsection{Green and Poisson kernel}
	Since we assume \ref{H1} throughout the article, we have $\int_{\R^d}e^{-t\Psi(\xi)}|\xi|^nd\xi<\infty$, for $t>0$ and $n\in\N$, see \cite[Eq. (3.5)]{vondra_heat}, so $X$ has transition densities  $p(t,x,y)=p(t,y-x)$ given by
	\begin{align*}
		p(t,x)=\frac{1}{(2\pi)^{d}}\int_{\R^d} \cos(x\cdot  \xi) e^{-t\Psi(\xi)}\, d\xi,\quad t>0,\,x\in\R^d.
	\end{align*}
	
	We assume that the process $X$ is transient, i.e. $\p_x(\lim_{t\to\infty}|X_t|=\infty)=1$, $x\in \R^d$. When $d\ge 3$ this  is always true, and for $d=2$  by the Chung-Fuchs condition this means that 
	\begin{align*}
		\int_0^1\frac{1}{\phi(\lambda)}d\lambda<\infty.
	\end{align*}  
	
	We define the potential kernel of $X$, i.e. the Green function of $X$, by
	\begin{align*}
		G(x)\coloneqq\int_{0}^\infty p(t,x)dt, \quad x\in \R^d.
	\end{align*}
	The kernel $G$ is the density of the mean occupation time for $X$, i.e. for $f\ge0$ we have
	\begin{align*}
		\int_{\R^d}G(x-y)f(y)dy=\mathbb{E}_x\left[\int_{0}^{\infty}f(X_t)dt\right],\quad x\in\R^d.
	\end{align*}
	 From \cite[Lemma 3.2(b)]{vondra_global} it follows that for every $M>0$ there is a constant $C=C(\phi,M)>0$ such that 
	\begin{align}\phantomsection\label{eq:G comparability}
		C^{-1} \frac{1}{|x|^d\phi(|x|^{-2})}\le G(x) \le	C \frac{1}{|x|^d\phi(|x|^{-2})},\quad |x|\le M.
	\end{align}
	In particular, $G$ is finite for $x\ne 0$. 
	Further, $G$ is rotationally symmetric and radially decreasing so we will slightly abuse notation by denoting $G(x,y)=G(x-y)=g(|x-y|)$. 
	
	For an open $D\subset \R^d$ set $\tau_D=\inf\{t>0: X_t\notin D\}$. We define the killed process $X^D$ by
	\begin{align*}
		X_t^D\coloneqq\begin{cases}
			X_t,&t<\tau_D,\\
			\partial,&t\ge \tau_D,
		\end{cases}
	\end{align*}
	where $\partial$ is an adjoint point to $\R^d$ called the cemetery. The process $X^D$ has a transition density which is for $t>0$ and  $x,y\in \R^d$  given by
	\begin{align}\phantomsection\label{eq:pD transition density}
		p^D(t,x,y)=p(t,x,y)-\mathbb{E}_x[p(t-\tau_D,X_{\tau_D},y)\mathbf{1}_{\{\tau_D<t\}}].
	\end{align}
	 It follows that $0\le p^D\le p$ and by repeating the proof of \cite[Theorem 2.4]{chung_zhao} we get that $p^D$ is symmetric. Since the process $X$ has right continuous paths, it follows that $p^D(t,x,y)=0$ if $x\in \overline D^c$ or $y\in \overline D^c$.  The Green function of $X^D$ is defined by $G_D(x,y)\coloneqq\int_0^\infty p^D(t,x,y)dt$, $x,y\in \R^d$, which is the density of the mean occupation time for $X^D$, i.e. for $f\ge0$ we have
	\begin{align}\phantomsection\label{eq:GDf expectation}
		\int_{D}G_D(x,y)f(y)dy=\mathbb{E}_x\left[\int_{0}^{\tau_D}f(X_t)dt\right],\quad  x\in \R^d.
	\end{align}
	Note that $G=G_{\R^d}$.
	
	For $x\in\R^d$ the $\mathbb{P}_x$ distribution of $X_{\tau_D}$ is denoted by $\omega_D^x$, i.e.
	\begin{align*}
		\mathbb{P}_x(X_{\tau_D}\in A)=\omega_D^x(A),\quad A\in\mathcal{B}(\R^d).
	\end{align*}
	The measure $\omega_D^x$ is concentrated on $D^c$ and since we are in  the transient case, we have  the following formula for  $x,y\in \R^d$
	\begin{align}\phantomsection\label{eq:G SMP}
		G_D(x,y)=G(x,y)-\mathbb{E}_x[G(X_{\tau_D},y)]=G(x,y)-\int_{D^c}G(w,y)\omega_D^x(dw).
	\end{align}
	It follows from \eqref{eq:pD transition density} that $G_D$ is symmetric and non-negative. On $(D\times D)\setminus\{(x,x):x\in D\}$ the kernel $G_D$ is jointly continuous which can be easily seen via well-known representation of the densities $p(t,x)$ in \cite[Eq. (2.8)]{mimica}.
	Using the strong Markov property and \eqref{eq:G SMP} one can easily show that for all open $U\subset D$ and $x,y\in \R^d$  it holds 
	\begin{align}\phantomsection\label{eq:GD SMP}
		G_D(x,y)=G_U(x,y)+\int_{U^c}G_D(w,y)\omega_U^x(dw).
	\end{align}
	Equation \eqref{eq:pD transition density} also yields that $G_D(x,y)=0$ if $x\in \overline D^c$ or $y\in \overline D^c$. Furthermore, if $y\in\partial D$, then $G_D(x,y)=0$ for all $x\in D$ if and only if $y$ is a {\it regular} point for $D$. A point $x\in \partial D$ is regular for $D$ if $\p_x(\tau_D=0)=1$, i.e. if $\omega^x_D=\delta_x$. A point at $\partial D$ which is not regular is called {\it irregular} and it is well-known that the set of irregular points is polar. This property will be used many times throughout the article.
	
	Equation \eqref{eq:GD SMP} yields that for every $x,y\in\R^d$ and  open   $U\subset D$ we have $G_U(x,y)\le G_D(x,y)$. In fact, if  we have open sets  $D_1\subset D_2\subset \dots \subset D$ and $\cup_n D_n=D$, then $G_{D_n}(x,y)\uparrow G_D(x,y)$, for every $x,y\in \R^d$ except if $x$ or $y$ are irregular for $D$. This follows from \eqref{eq:G SMP},  the  continuity of $G$  off the diagonal  and  the  quasi-left-continuity of $X$.

	For an open $D\subset \R^d$, we define $P_D$, the Poisson kernel of $D$ with respect to $X$, by
	\begin{align}\phantomsection\label{eq:PD defn}
		P_D(x,y)\coloneqq\int_DG_D(x,w)j(|w-y|)dw,\quad (x,y)\in \R^d\times D^c.
	\end{align}
	If $x\in D$ the measure $\omega_D^x$ is absolutely continuous with respect to the Lebesgue measure in the interior of $D^c$. Its Radon-Nikodym derivative is $P_D(x,\cdot)$, see \cite[Eq. (1.1)]{vondra_fm_2016}. Further, if the boundary of $D$ possesses enough regularity, e.g. if $D$ is a Lipschitz set, then
	\begin{align}\phantomsection\label{eq:PD Lipschitz set}
		\omega_D^x(dy)=P_D(x,y)dy,\quad\text{ on the whole $D^c$,}
	\end{align}
	see \cite[Proposition 4.1]{millar1975}.
	
	By integrating \eqref{eq:GD SMP} with respect to $j(|y-z|)dy$ on $\R^d$, with $z\in D^c$, and using Fubini's theorem we get
	\begin{align}\phantomsection\label{eq:PD SMP}
		P_D(x,z)=P_U(x,z)+\int_{D\setminus U}P_D(w,z)\omega_U^x(dw),\quad (x,z)\in U\times D^c,
	\end{align}
	where we used that the sets of irregular points at $\partial D$ and $\partial U$ are polar.

	\begin{defn}\phantomsection\label{d:Green integral}
		Let $D\subset \R^d$ be an open set and $f:D\to[-\infty,\infty]$. The Green potential of $f$ is defined by
		\begin{align}\phantomsection\label{eq:def GDf}
			G_Df(x)&\coloneqq\int_D G_D(x,y)f(y)dy,
		\end{align}
		for all $x\in\R^d$ such that the integral above converges absolutely.
	\end{defn}
	
	\begin{lem}\phantomsection\label{r:konacnost od GDf}
		Let $f\ge 0$. If the integral $\int_DG_D(x_0,y)f(y) dy$  converges at one point $x_0\in D$, then  $G_Df<\infty$ a.e., $G_Df\in \mathcal{L}^1$ and $f\in L^1_{loc}(D)$.
	\end{lem}		
	\begin{proof}
		Let $0< s<\delta_D(x_0)$, and denote just for this  proof  $B=B(x_0, s)$.  Using the strong Markov property  we have
		\begin{align}
			\begin{split}\phantomsection\label{eq:GD finite}
				\infty>G_Df(x_0)&
				\ge\mathbb{E}_{x_0}\left[\int\limits_{\tau_B}^{\tau_D}f(X_t)dt\right]=\mathbb{E}_{x_0}\left[\mathbb{E}_{X_{\tau_B}}\left[\int\limits_0^{\tau_D}f(X_t)dt\right]\right]\\
				&=\mathbb{E}_{x_0}[G_Df(X_{\tau_B})]=\int_{B^c}G_Df(y)P_B(x_0,y)dy.
			\end{split}
		\end{align}
		\noindent From  \cite[Lemma 2.2]{grzywny_potential_kernels} we have that $P_{ B}(x_0,y)\ge  c_1 j(|x_0-y|)$, $y\in \overline{B}^c$. Furthermore, let $r_0\in(1,\infty)$ such that $j(|y|)\le 1$, for $|y|\ge r_0$. Inequality \eqref{eq:scm_2.12} implies that there is a constant $ c_2>0$ such that	$j(|y|)\le c_2 j(|x_0-y|)$, for all $|y|\ge r_0$. Let $m\coloneqq \inf\{j(|x_0-y|):y\in B^c, |y|\le r_0\}>0$. Thus, for $y\in \overline B^c$ we have
		\begin{align*}
			1\wedge j(|y|)
			\le \max\{c_2,1/m\} j(|x_0-y|).
		\end{align*}
		Therefore, there is $c_3>0$ such that $P_B(x_0,y)\ge c_3 (1\wedge j(|y|))>0$, $y\in \overline B^c$.  This yields $$\int_{B^c}G_Df(y)(1\wedge j(|y|))dy<\infty,$$ hence $G_Df<\infty$ a.e. on $B^c$. Starting the calculations again from the point $\tilde x\in D\setminus \overline B$ such that $G_Df(\tilde x)<\infty$, we also get $\int_{B}G_Df(y)(1\wedge j(|y|))dy<\infty$. Hence, $G_Df<\infty$ a.e. and $G_Df\in \mathcal{L}^1$. 
		
		To prove that $f\in L^1_{loc}(D)$ take $U\subsub D$ and $x\in D\setminus \overline U$ such that $G_Df(x)<\infty$. Since the function $y\to G_D(x,y)$ is bounded from below and above on $U$ by the Harnack inequality, see \cite[Theorem 7]{grzywny_harnack}, we have the claim.
		
	\end{proof}
	
		The following proposition is an extension of \cite[Lemma 5.3]{bogdan2000} to more general non-local operators.
	\begin{prop}\phantomsection\label{p:L(Gf)=-f}
		 Let $D$ be an open set.  If $f:D\to[-\infty,\infty]$ satisfies $G_D|f|(x)<\infty$ for some $x\in D$, then $\tilde{L}(G_Df)=-f$ in $D$.
	\end{prop}
	\begin{proof}
		 In \cite[Lemma 3.5]{moritz} the claim was proved for bounded $D$ and for $f\in L^1(D)$. Recall that  Lemma  \ref{r:konacnost od GDf} yields that $G_Df$ is well defined almost everywhere and  $f\in L^1_{loc}(D)$. Without loss of generality we can assume that $f\ge0$. 
		
		Suppose that $D$ is bounded and $f\in L^1_{loc}(D)$. There is an increasing sequence of precompact sets $(K_n)_n$ in $D$ such that $\cup_{n}K_n=D$. Define $f_n\coloneqq f\mathbf{1}_{K_n}\in L^1$. Obviously, $G_Df_n\uparrow G_Df$  a.e.  and also in $\mathcal{L}^1$ due to  Lemma \ref{r:konacnost od GDf}  and the dominated convergence theorem. Hence,  due to \eqref{e: L C_c}  for all $\varphi\in C^\infty_c(D)$ we get
		\begin{align*}
			\langle \tilde{L}G_Df,\varphi\rangle&=\langle G_Df,L\varphi\rangle=\lim_{n\to\infty} \langle G_Df_n,L\varphi\rangle\\
			&=\lim_{n\to\infty}- \langle f_n,\varphi\rangle= -\langle f,\varphi\rangle.
		\end{align*}
		Now take $D$ unbounded and $f\in L^1_{loc}(D)$. There is an increasing sequence of open precompact sets $(D_n)_n$ in $D$ such that $\cup_{n}D_n=D$. Obviously, $G_{D_n}f\uparrow G_Df$  a.e.  and in $\mathcal{L}^1$. Take any $\varphi\in C^\infty_c(D)$. There is $n_0\in \N$ such that for all $n\ge n_0$ we have  $\varphi\in C^\infty_c(D_n)$. Hence,
		\begin{align*}
			\langle \tilde{L}G_Df,\varphi\rangle&=\langle G_Df,L\varphi\rangle=\lim_{n\to\infty} \langle G_{D_n}f,L\varphi\rangle\\
			&=\lim_{n\to\infty}- \langle f\mathbf{1}_{D_n},\varphi\rangle= -\langle f,\varphi\rangle.
		\end{align*}
	\end{proof}

	\section{Poisson kernel and $L$-harmonic functions}\phantomsection\label{s:Poisson and harmonic}
	
	\begin{prop}\phantomsection\label{Poisson_continuity}
		Let $D$ be an open set. Then $P_D:D\times\overline{D}^c\to(0,\infty)$ is jointly continuous.
	\end{prop}
	\begin{proof}
		We imitate the proof of the similar claim for the isotropic $\alpha$-stable process, see \cite[Theorem 5.7]{vondra_2002}. Let $(x_n)_n\subset D$ and $(z_n)_n\subset \overline{D}^c$ such that $x_n\to x\in D$, and $z_n\to z\in\overline{D}^c$. Let $0<\varepsilon,\delta<1$ such that $\delta_D(x)>2\delta$ and $\delta_{D^c}(z)>2\varepsilon$. Then for all large enough $n\in\N$ we have $\delta_D(x_n)>\delta$ and $\delta_{D^c}(z_n)>\varepsilon$. We have  by \eqref{eq:PD defn}
		\begin{align*}
			&|P_D(x_n,z_n)-P_D(x,z)|=
			\left|\int\limits_D G_D(x_n,y)j(|y-z_n|)dy-\int\limits_D G_D(x,y)j(|y-z|)dy\right|\\
			&\quad\le\left|\int\limits_{D\cap {B(x,2\delta)}^c} G_D(x_n,y)j(|y-z_n|)dy-\int\limits_{D\cap {B(x,2\delta)}^c}  G_D(x,y)j(|y-z|)dy\right|\\
			&\qquad\quad+\int\limits_{B(x,2\delta)} G_D(x_n,y)j(|y-z_n|)dy+\int\limits_{B(x,2\delta)} G_D(x,y)j(|y-z|)dy.
		\end{align*}
		 Recall that $j$ is continuous and that $G_D$ is continuous off the diagonal. Thus,  for the first term  we have by the dominated convergence theorem
		\begin{align*}
			\lim_{n\to \infty}\int\limits_{D\cap {B(x,2\delta)}^c} G_D(x_n,y)j(|y-z_n|)dy=\int\limits_{D\cap {B(x,2\delta)}^c}  G_D(x,y)j(|y-z|)dy.
		\end{align*}
		Indeed,  we can apply the dominated convergence theorem since 
		$G_{\R^d}$ is radially decreasing so there is $c_1>0$ such that $G_D(w,y) \le G_{\R^d}(w,y)\le  c_1$ for all $w\in B(x,\delta)$ and $y\in B(x,2\delta)^c$. Also, using  \eqref{eq:scm_2.12} there is $c_2>0$ such that $j(|y-q|)\le c_2 j(|y-z|)$ for $q\in B(z,\varepsilon)$ and $y\in D\cap B(x,2\delta)^c$.
		
		For the other two integrals we use the estimate \eqref{eq:G comparability}, i.e. we use
		\begin{align*}
			G_D(x,y)\le G_{\R^d}(x,y) \le c_3\,\frac{1}{|x-y|^d \phi(|x-y|^{-2})}, \quad |x-y|<3,
		\end{align*}
		where $c_3=c_3(\phi)>0$. Now for all $w\in B(x,\delta)$, and $q\in B(z,\varepsilon)$ we have
		\begin{align*}
			\int\limits_{B(x,2\delta)} &G_D(w,y)j(|y-q|)dy\le j(\varepsilon)\int\limits_{B(x,2\delta)} G_D(w,y)dy\\
			&\le j(\varepsilon)\left(\int\limits_{B(x,2\delta)\cap B(w,\delta)} G_D(w,y)dy+\int\limits_{B(x,2\delta)\cap B(w,\delta)^c} G_D(w,y)dy\right)\\
			&\le j(\varepsilon)c_3\left(\int\limits_{0}^{\delta}\frac{dr}{r\phi(r^{-2})}+\int_\delta^{3\delta}\frac{dr}{r\phi(r^{-2})}\right)	\le j(\varepsilon)c_3\left(\int\limits_{0}^{3\delta}\frac{dr}{r\phi(r^{-2})}\right) \overset{\delta\to0}{\longrightarrow}0,
		\end{align*} 
		where we use \eqref{LSC} for  the  convergence of the integral part.
	\end{proof}

	\begin{defn}\phantomsection\label{Poisson integral}
		Let $D\subset \R^d$ be an open set and let
		$\lambda$ be a $\sigma$-finite signed measure on $D^c$ such that for all $x\in D$
		\begin{align}\phantomsection\label{eq:konacnost tocke PDlambda}
			\int_{D^c} P_D(x, y)|\lambda|(dy) < \infty.
		\end{align}
		The Poisson integral of $\lambda$ is defined by
		\begin{align*}
			P_D\lambda(x)&\coloneqq\int_{D^c} P_D(x,y)\lambda(dy),\quad x\in D.
		\end{align*}
	\end{defn}
	
	We extend the definition of the Poisson integral for non-negative $\sigma$-finite measures by the same formula, i.e. for $\sigma$-finite measure $\lambda$ we define
	\begin{align*}
		P_D\lambda(x)&\coloneqq\int_{D^c} P_D(x,y)\lambda(dy)\in[0,\infty],\quad x\in D.
	\end{align*}
	Although this seems as an extension of the definition, it will follow from Theorem \ref{t:thm_1_1} that either $P_D|\lambda|\equiv \infty$ or $P_D|\lambda|<\infty$ in $D$, see Remark \ref{r:konacnost tocke PDlambda}. 
	
	It will be of considerable interest to extend $P_D\lambda$ to the whole $\R^d$ in the following sense. We define the  (signed)  measure $P_D^*\lambda$ by
	\begin{align}\phantomsection\label{eq:PD* definition}
		P_D^*\lambda(dy)=P_D\lambda(y)\mathbf{1}_D(y)dy+\mathbf{1}_{D^c}(y)\lambda(dy),
	\end{align}
	i.e.  $P_D^*\lambda$ is on $D$  the (signed)  measure with the density function $P_D\lambda$ and on $D^c$ it is the  (signed)  measure $\lambda$. This extension was introduced in \cite[Eq. (25)]{bogdan_est_and_struct} for the case of the isotropic $\alpha$-stable process.
	\begin{rem}\phantomsection\label{r:def Poisson integral}
		Suppose that $P_D|\lambda|(x)<\infty$ for all $x\in D$. Then $\lambda$ is finite on compact subsets of $\overline D^c$. Indeed, let $K$ be a compact subset of $\overline{D}^c$ and let $s\in(0,1)$ such that $\overline{B(x,s)}\subset D$.  For $y\in \overline D^c$ by \cite[Proposition 4.7]{vondra_2012_SCM} we have $P_D(x,y)\ge P_{B(x,s)}(x,y)\ge c_1\,j(|x-y|)$, where $c_1>0$.  Thus, since $j$ is continuous and strictly positive, we have 
		$$ \infty> \int_{D^c}P_D(x,y)|\lambda|(dy)\ge  c_2 \,|\lambda|(K),$$
		 where $c_2>0$. 
		Furthermore, in Remark \ref{r:acc} we will see that $\lambda$ can have some mass on $\partial D$ but only on  the  specific part of the boundary at so-called inaccessible points.
	\end{rem}

	\begin{lem}\phantomsection\phantomsection\label{lemma 7}
		\begin{enumerate}[(a)]
			\item
			Let $R\in(0,1)$. There is a constant $C=C(\phi)>0$ such that if $\lambda$ is a $\sigma$-finite measure supported on $B_R^c$, and $D\subset B_R$, then for all $x\in D\cap B_{R/2}$ it holds
			\begin{align}\phantomsection\label{(39)}
				C^{-1}\,\mathbb{E}_x\tau_D\int\limits_{B_{R/2}^c}j(|y|)P_D^*\lambda(dy)\le P_D\lambda(x)\le C\,\mathbb{E}_x\tau_D\int\limits_{B_{R/2}^c}j(|y|)P_D^*\lambda(dy).
			\end{align}
			\item
			Suppose \ref{H2} and let $R\ge 1$. There is a constant $C=C(\phi)>0$ such that if $\lambda$ is a $\sigma$-finite measure supported on $\overline{B}_R$, and $D\subset \overline{B}_R^c$, then for all $x\in D\cap \overline{B}_{2R}^c$ it holds
			\begin{align}\phantomsection\label{(39) - infty}
				C^{-1}P_D(x,0)\int\limits_{\overline B_{2R}}P_D^*\lambda(dy)\le P_D\lambda(x)\le C\,P_D(x,0)\int\limits_{\overline B_{2R}}P_D^*\lambda(dy).
			\end{align}
		\end{enumerate}
	\end{lem}
	
	\begin{proof}
		For part $(a)$ we will use \cite[Lemma 5.4]{vondra_2012_SCM}. Notice that the inequality from the statement of \cite[Lemma 5.4]{vondra_2012_SCM} is valid for any $(x,y)\in (U\cap B(z_0,r/2))\times B(z_0,r)^c$. This can be seen by inspecting the proof of the lemma since  \cite[Eq. (5.1)]{vondra_2012_SCM} can be extended to \eqref{eq:PD SMP}. Hence, to finish the proof we just need to integrate the mentioned inequality with respect to the measure $\lambda(dy)$, where $z_0=0$, $U=D$ and $r=R$.
		
		For part $(b)$ we will use \cite[Lemma 3.4]{vondra_bhpinf}. Similarly as above, the inequality from the statement of \cite[Lemma 3.4]{vondra_bhpinf} is valid for any $(x,z)\in (U\cap \overline{B(0,ar)}^c)\times \overline{B(0,r)}$ which can be checked by inspecting the proof. Again, the only difference is in the fact that  \cite[Eq. (3.10)]{vondra_bhpinf} can be extended to \eqref{eq:PD SMP}. To finish the proof we need to integrate the mentioned inequality with respect to the measure $\lambda(dz)$ where $a=2$, $U=D$ and $r=R$. 
	\end{proof}
	
	Lemma \ref{lemma 7} yields the following version of a uniform boundary Harnack principle.
	\begin{thm}\phantomsection\label{t:thm_1_1}
		\begin{enumerate}[(a)]
			\item
			There is a constant $C=C(\phi)>1$ such that for every $R\in(0,1)$, for all open $D\subset\R^d$, $x_1,x_2\in D\cap B_{R/2}$, $y_1,y_2\in D^c\cap B_{R}^c$, and for all $\sigma$-finite measures $\rho$, $\lambda$ on $B_R^c$ we have
			\begin{align}\phantomsection\label{44_18 - 1}
				P_D(x_1,y_1)P_D(x_2,y_2)\le C\,P_D(x_1,y_2)P_D(x_2,y_1)
			\end{align}
			and
			\begin{align}\phantomsection\label{44_18}
				P_D\rho(x_1)P_D\lambda(x_2)\le C\,P_D\rho(x_2)P_D\lambda(x_1).
			\end{align}
			\item
			Suppose \ref{H2}. There is a constant $C=C(\phi)>1$ such that for every $R\ge1$, for all open $D\subset\R^d$, $x_1,x_2\in D\cap \overline{B}_{2R}^c$, $y_1,y_2\in D^c\cap \overline{B}_{R}$, and for all $\sigma$-finite measures $\rho$, $\lambda$ on $\overline{B}_R$ we have
			\begin{align}
				P_D(x_1,y_1)P_D(x_2,y_2)\le C\,P_D(x_1,y_2)P_D(x_2,y_1)
			\end{align}
			and
			\begin{align}\phantomsection\label{44_18 - infty}
				P_D\rho(x_1)P_D\lambda(x_2)\le C\,P_D\rho(x_2)P_D\lambda(x_1).
			\end{align}
		\end{enumerate}
	\end{thm}
	The first part of this theorem is an extension of \cite[Theorem 1.1(ii)]{vondra_2012_SCM} with $\overline{D}^c$ being replaced by $D^c$, i.e. the difference is that points $y_1$ and $y_2$ in \eqref{44_18 - 1} can be at $\partial D$. This subtle difference comes as a consequence of Lemma \ref{lemma 7} and will play a very important role in proving the results on the relative oscillation of Poisson integrals, e.g. Lemma \ref{lemma 8}.
	\begin{proof}[Proof of Theorem \ref{t:thm_1_1}]
		We give the proof of the first claim. The second claim follows similarly.
		
		Let $D_R=D\cap B_R$. It is easy to see from \eqref{eq:PD SMP} that for $x_i\in D\cap B_{R/2}$, $i\in\{1,2\}$, and $y_j\in D^c\cap B_R^c$, $j\in\{1,2\}$, we have that $P_D(x_i,y_j)=P_{D_R}\lambda_j(x_i)$ for  some  measure $\lambda_j$ supported on $B_R^c$. Now \eqref{44_18 - 1} follows from Lemma \ref{lemma 7}. By integrating \eqref{44_18 - 1} with respect to the measures $\rho(dy_1)$ and $\lambda(dy_2)$ we get \eqref{44_18}.
	\end{proof}

	\begin{rem}\phantomsection\label{r:konacnost tocke PDlambda}
		 Note that for the $\sigma$-finite measures $\rho$ and $\lambda$ appearing in Theorem \ref{t:thm_1_1} we do not assume \eqref{eq:konacnost tocke PDlambda}. However,  by fixing $\rho=\delta_{y_2}$, where $y_2\in \overline{D}^c$, it follows from \eqref{44_18} that if for a $\sigma$-finite signed measure $\lambda$  on $D^c$  we have $P_D|\lambda|(x)<\infty$ for some $x\in D$, then we have $P_D|\lambda|(x)<\infty$ for all $x\in D$. This means that  either $P_D|\lambda|\equiv\infty$ or $P_D|\lambda|<\infty$ in $D$. 
	\end{rem}

	Before we define $L$-harmonic functions we recall that a function $u:\R^d\to\R$  is said to be harmonic with respect to the process $X$ in an open set $D\subset \R^d$ if for every open $U\subsub D$  and all $x\in U$ it holds that $\ex_x[|u(X_{\tau_U})|]<\infty$ and 
	\begin{align}\phantomsection\label{eq:defn of obicna harm}
		u(x)=\ex_x[u(X_{\tau_U})].
	\end{align}
	We say that $u$ is regular harmonic in $D$ if \eqref{eq:defn of obicna harm} holds with $U=D$. If $u$ is harmonic in $D$ and $u=0$ in $\overline{D}^c$, then $u$ is said to be singular harmonic.
	From \eqref{eq:GD SMP} we can see that for $y\in D$ the function $x\mapsto G_D(x,y)$ is harmonic in $D\setminus \{y\}$ and regular harmonic in $D\setminus B(y,\varepsilon)$ for every $\varepsilon>0$.
	
	\begin{defn}\phantomsection\label{d:L-harmonic function}
		Let $D\subset \R^d$ be an open set. We say that $f:D\to\R$ is $L$-harmonic in D with outer charge $\lambda$ if $\lambda$ is a $\sigma$-finite (signed) measure on $D^c$ and if for every $U\subset\subset D$ and $x\in U$ we have
		\begin{align}\phantomsection\label{mean-value-property}
			f(x)
			=\int_{D^c}P_U(x,y)\lambda(dy)+\int_{D\setminus U}f(y)\omega_U^{x}(dy),
		\end{align}
		where the integrals converge absolutely.
	\end{defn}
	
	The definition  above  was  first  used in \cite{bogdan_est_and_struct} for the isotropic $\alpha$-stable process with an additional assumption of  continuity of the function $f$. We prove in Proposition \ref{p:neprekidnost harmonijskih} that this additional assumption can be dropped, and in Theorem \ref{r:neprekidnost harmoncijskih Cinfty} we prove that $f\in C^\infty(D)$. Furthermore, note that a function $u:\R^d\to\R$ which is harmonic in $D$ is $L$-harmonic in $D$ with outer charge $\lambda(dy)=u(y)dy$. Indeed, take $U\subsub D$ and $x\in U$. Equation \eqref{eq:defn of obicna harm} implies
	\begin{align}\phantomsection\label{eq:harm <-> L-harm}
		\begin{split}
			u(x)&=\ex_x[u(X_{\tau_U})]=\int_{U^c}u(y)\omega_U^{x}(dy)\\
			&=\int_{D^c}P_U(x,y)u(y)dy+\int_{D\setminus U}u(y)\omega_U^{x}(dy),
		\end{split}
	\end{align}
	where we used that $P_U(x,\cdot)$ is the density of $\omega_U^{x}$ in the interior of $U^c$. Hence, every harmonic function is $L$-harmonic. 
	 Furthermore,  if $u$ is $L$-harmonic in $D$ with outer charge $\lambda$ such that $\lambda$ is absolutely continuous with respect to the Lebesgue measure on $D^c$, then $u$ is harmonic in $D$. In particular, if $u$ has zero outer charge, i.e. $\lambda\equiv 0$, then $u$ is a singular harmonic function.

	If $f$ is $L$-harmonic in $D$ with outer charge $\lambda$ we sometimes abbreviate notation by saying $(f,\lambda)$ is $L$-harmonic in $D$. Property \eqref{mean-value-property} is often referred to as the mean-value property because of the connection with \eqref{eq:harm <-> L-harm}. Similarly as in \eqref{eq:PD* definition}, integrating with respect to $(f,\lambda)$ means that we integrate with respect to the measure $f(y)\mathbf{1}_D(y)dy+\mathbf{1}_{D^c}(y)\lambda(dy)$. We continue with a few properties of $L$-harmonic functions.
	
	\begin{lem}\phantomsection\label{l:harmonijske su u L^1}
		 Let $D$ be an open set.  If $(f,\lambda)$ is $L$-harmonic in $D$, then
		\begin{align}\phantomsection\label{harm in L^1 - nejednakost}
			 \int_{\R^d}(1\wedge j(|y|))(|f|,|\lambda|)(dy)<\infty.
		\end{align}
		In particular, $f\in L^1_{loc}(D)$ and if $D$ is bounded, we have $f\in L^1(D)$.
	\end{lem}
	\begin{proof}
		Let $B(x,s)\subsub D$ with $s<1$. From \eqref{mean-value-property} with $U=B(x,s)$  and with the same calculations as in  Lemma  \ref{r:konacnost od GDf} we get that
		\begin{align*}
			\infty>\int_{D^c}(1\wedge j(|y|))\lambda(dy)+\int_{D\setminus B(x,s)}f(y)(1\wedge j(|y|))(dy).
		\end{align*}
		Since $j$ is continuous and $j>0$ we see that $f\in L^1_{loc}(D)$ and we get \eqref{harm in L^1 - nejednakost}. 
		Obviously, if $D$ is bounded, we have $f\in L^1(D)$.
	\end{proof}
	\begin{prop}\phantomsection\label{p:neprekidnost harmonijskih}
		 Let $D$ be an open set.  If $(f,\lambda)$ is $L$-harmonic in $D$, then $f\in C(D)$.
	\end{prop}
	\begin{proof}
		Let $x\in D$ and $(x_n)_n\subset D$ such that $x_n\to x$. Let $0<\varepsilon<1$ be such that $\delta_D(x)>\varepsilon$. Without loss of generality suppose that for all $n\in\N$ we have $x_n\in B(x,\varepsilon/2)$. Using \eqref{mean-value-property} with $U=B(x,\varepsilon)$ and applying \eqref{eq:PD Lipschitz set} we have
		\begin{align*}
			f(x)&=\int_{D^c}P_{B(x,\varepsilon)}(x,y)\lambda(dy)+\int_{D\setminus B(x,\varepsilon)}f(y)P_{B(x,\varepsilon)}(x,y)dy,\\
			f(x_n)&=\int_{D^c}P_{B(x,\varepsilon)}(x_n,y)\lambda(dy)+\int_{D\setminus B(x,\varepsilon)}f(y)P_{B(x,\varepsilon)}(x_n,y)dy.
		\end{align*} 
		Note that Proposition \ref{Poisson_continuity} yields $P_{B(x,\varepsilon)}(x_n,y)\to P_{B(x,\varepsilon)}(x,y)$. Also, inequality \eqref{44_18 - 1} for $D=B(x,\varepsilon)$ implies that there is a constant $c>0$ such that $P_{B(x,\varepsilon)}(x_n,y)\le c\, P_{B(x,\varepsilon)}(x,y)$, for all $n\in\N$ and all $y\in B(x,s)^c$.  Now by the dominated convergence theorem we have $f(x_n)\to f(x)$.
	\end{proof}

	 In Theorem \ref{r:neprekidnost harmoncijskih Cinfty} we strengthen the previous proposition by proving that $f\in C^\infty(D)$. This is achieved using the same technique as in \cite[Proposition 3.2 \& Theorem 1.7]{grzywny_potential_kernels}. First we invoke \cite[Proposition 3.2]{grzywny_potential_kernels} and its consequences.

	\begin{lem}\phantomsection\label{l:radijalna}
		Let $0\le q <r<\infty$. There is a radial kernel function $\overline P_{q,r}:\R^d\to\R$, a constant $C=C(\phi,q,r)>0$, and a probability measure $\mu_{q,r}$ on $[q,r]$ with the following properties:
		\begin{enumerate}[(a)]
			\item $0\le \overline P_{q,r}\le C$ in $\R^d$, $\overline P_{q,r}=0$ in $B_q$, $\overline P_{q,r}=C$ in $B_r\setminus B_q$, $\overline P_{q,r}$ is radially decreasing, and $\overline P_{q,r}(z)\le P_{B_r}(0,z)$, for $|z|>r$;
			\item
			for any $A\in\BB(\R^d)$ it holds
			\begin{align}\phantomsection\label{eq:mean of radial kernel}
				\int_A\overline P_{q,r}(z)dz=\int_{[q,r]}\int_A P_{B_s}(0,y)dy\,\mu_{q,r}(ds).
			\end{align}
		\end{enumerate}
	\end{lem}
	Equality \eqref{eq:mean of radial kernel} implies the following claim.
	\begin{lem}\phantomsection\label{l:mean radijalna}
		Let $0\le q <r<\infty$ and $\varepsilon>0$. If $(f,\lambda)$ is $L$-harmonic in $B_{r+\varepsilon}$, then
		$$ f(0)=\int_{\R^d\setminus B_q}\overline P_{q,r}(z)(f,\lambda)(dz).$$
		In particular, if $(f,\lambda)$ is $L$-harmonic in $B_{2r+\varepsilon}$, then for $x\in B_r$ it holds
		$$ f(x)=\int_{\R^d\setminus B(x,q)}\overline P_{q,r}(x-z)(f,\lambda)(dz),$$
		i.e. $f=(f,\lambda)\ast\overline P_{q,r}$ in $B_r$.
	\end{lem}
	The following theorem is a generalization of \cite[Theorem 1.7]{grzywny_potential_kernels} to $L$-harmonic functions.
	\begin{thm}\phantomsection\label{r:neprekidnost harmoncijskih Cinfty}
		Let $D$ be an open set. If $f$ is $L$-harmonic in  $D$ with outer charge $\lambda$, then $f\in C^\infty(D)$.
	\end{thm}
	\begin{proof}
		The claim can be proved in the same way as in \cite{grzywny_potential_kernels}. However, in \cite{grzywny_potential_kernels} it was assumed that $f$ is bounded so we will repeat and slightly extend the first part of the proof to justify the calculations that follow.
		
		Due to the translation invariance of the process $X$, we can assume that $0\in D$.	Let $r\in(0,1)$ and $k\in \N$ be such that $B_{2(k+1)r}\subsub D$. Set $q=0$ and let $C_{r}$ denote $C(\phi,0,r)>0$ of Lemma \ref{l:radijalna}. Further, let $\kappa$ be a non-negative smooth radial function which takes values in $[0,1]$, which is equal to $1$ in $B_{3r/2}$, and which is equal to 0 in $B^c_{2r}$. Define $\pi_r(z)=\overline P_{0,r}(z)\kappa(z)$, and $\Pi_r(z)=\overline P_{0,r}(z)(1-\kappa(z))$. 
		
		Note that Proposition \ref{p:neprekidnost harmonijskih} yields that $f$ is bounded on $B_{2(k+1)r}$ so set $m:=\sup_{B_{2(k+1)r}}|f|<\infty$.	From Lemma \ref{l:radijalna}$(a)$ for $x\in B_{2kr}$ we have
		\begin{align}\phantomsection\label{eq:aps radijalna}
			(|f|,|\lambda|)\ast \overline P_{0,r}(x)&=\int\limits_{B(x,2r)}|f(z)| \overline P_{0,r}(x-z)dz+\int\limits_{B^c(x,2r)}\overline P_{0,r}(x-z)(|f|,|\lambda|)(dz)\notag\\
			&\le c_1\,  m\,C_{r} + \int\limits_{B^c(x,2r)}P_{B(x,r)}(x,z)(|f|,|\lambda|)(dz),
		\end{align}
		where $c_1=c_1(r)>0$ is the volume of a ball with radius $r$. 
		From \cite[Proposition 4.7]{vondra_2012_SCM} we have $P_{B(x,r)}(x,z)\le c_2 j(|x-z|-r)$, where $c_2=c_2(\phi,r)>0$. Thus, using \eqref{eq:scm_2.12}, we get that there is $c_3=c_3(\phi,k,r)>0$ such that for all $x\in B_{2kr}$ and $z\in B^c(x,2r)$ it holds
		$$P_{B(x,r)}(x,z)\le c_3 (1\wedge j(|z|)).$$
		Applying this inequality in \eqref{eq:aps radijalna} and recalling Lemma \ref{l:harmonijske su u L^1} we get that there is $M=M(\phi,k,r,f,\lambda)<\infty$ such that 
		\begin{align}\phantomsection\label{eq:radijalna ocjena}
			(|f|,|\lambda|)\ast \overline P_{0,r}\le M,\quad \textrm{in $B_{2kr}$}.
		\end{align}
		Obviously, since $\overline P_{0,r}=\pi_r+\Pi_r$,  we have $|f|\ast \pi_r\le M$ and $(|f|,|\lambda|)\ast \Pi_r\le M$ in $B_{2kr}$. Also, since $f=(f,\lambda)\ast \overline{P}_{0,r}$ in $B_{2kr}$, we have 
		\begin{align}\phantomsection\label{eq:radijal convolution}
			f=f\ast \pi_r + (f,\lambda)\ast \Pi_r,\quad\textrm{ in $B_{2kr}$}.
		\end{align}
		Finally, inequality \eqref{eq:radijalna ocjena} implies that the convolution property \eqref{eq:radijal convolution} can be used iteratively to get that for $x\in B_r$ it holds
		\begin{align*}
			f=(\delta_0 +\pi_r+\pi^{\ast2}_r+\dots,\pi^{\ast (k-1)}_r)\ast \Pi_r\ast (f,\lambda)+\pi^{\ast k}_r\ast f.
		\end{align*}
		Since all derivatives of the jumping kernel $j$ exist and are absolutely integrable in $B^c_\varepsilon$, for every $\varepsilon>0$, see \cite[Proposition 7.2]{bogdan_extension}, we may proceed with the proof in the same way as in \cite[Theorem 1.7]{grzywny_potential_kernels}.
	\end{proof}
	
	\begin{cor}\phantomsection\label{c:P_D harm}
		 Let $D$ be an open set.  If $\lambda$ is a $\sigma$-finite signed measure on $D^c$ satisfying \eqref{eq:konacnost tocke PDlambda}, then for every $x\in U\subset D$
		\begin{align}\phantomsection\label{eq:mvf for PD}
			P_D\lambda(x)=\int_{D^c}P_U(x,y)\lambda(dy)+\int_{D\setminus U}P_D\lambda(y)\omega_U^{x}(dy).
		\end{align}
		In particular, $P_D\lambda$ is $L$-harmonic in $D$ with outer charge $\lambda$ and $P_D\lambda\in  C^\infty(D)\cap L^1_{loc}(D)$. Also, if $D$ is bounded $P_D\lambda \in L^1(D)$.
	\end{cor}
	\begin{proof}
		Take $U\subset D$ and $x\in U$. By integrating \eqref{eq:PD SMP}  with respect to $\lambda(dz)$ we get \eqref{eq:mvf for PD}. In particular, $P_D\lambda$ is $L$-harmonic in $D$ with outer charge $\lambda$. Hence by  Theorem \ref{r:neprekidnost harmoncijskih Cinfty}  and Lemma \ref{l:harmonijske su u L^1} we have $P_D\lambda \in  C^\infty(D)\cap L^1_{loc}(D)$ and if $D$ is bounded, then $P_D\lambda\in L^1(D)$.
	\end{proof}

	\begin{rem}\phantomsection\label{r:L-harm are generalized harm}
		Note that \eqref{eq:mvf for PD} holds for every $U\subset D$ which is a lot stronger than needed in \eqref{mean-value-property}. This property will be heavily used in proving results on the relative oscillation of Poisson integrals.
	\end{rem}
	We finish this section by proving two theorems about the connection between harmonic functions and the operator $L$. First we prove an auxiliary result.

	\begin{lem}\phantomsection\label{l:L(Plambda)=0}
		 Let $D$ be an open set  and $\lambda$ be a $\sigma$-finite  signed  measure on $D^c$ such that \eqref{eq:konacnost tocke PDlambda} is satisfied. Then $\tilde{L}(P_D^*\lambda)=0$ in $D$.
	\end{lem}
	\begin{proof}
		First recall that for $\varphi\in C_c^\infty(D)$ we have
		\begin{align*}
			L\varphi (x)=\mathrm{P.V.}\int_{\R^d}(\varphi(y)-\varphi(x))j(|x-y|)dy,
		\end{align*}
		and
		\begin{align*}
			\langle\tilde{L}(P_D^*\lambda),\varphi\rangle=\langle P_D^*\lambda,L\varphi\rangle&=\int_DP_D\lambda(x)L\varphi(x)dx+\int_{D^c}L\varphi(x)\lambda(dx)\\
			&\eqqcolon I_1+I_2.
		\end{align*}
		Note that $P_D\lambda(x)=G_Df(x)$, $x\in D$, where $f(z)=\int_{D^c}j(|z-y|)\lambda(dy)$. For the integral $I_1$ by Proposition \ref{p:L(Gf)=-f} we have
		\begin{align*}
			\int_D P_D\lambda(x)L\varphi(x)dx=-\int_D \left(\int_{D^c}j(|x-y|)\lambda(dy)\right)\varphi(x)dx.
		\end{align*}
		For the integral $I_2$ recall that $\supp \varphi\subset D$ and $\varphi=0$ on $D^c$. Hence
		\begin{align*}
			\int_{D^c}L\varphi(x)\lambda(dx)&=\int_{D^c}\left(\int_{D}\varphi(y)j(|x-y|)dy\right)\lambda(dx)\\
			&=\int_{D}\varphi(y)\left(\int_{D^c}j(|x-y|)\lambda(dx)\right)dy,
		\end{align*}
		where we can change the order of integration by Fubini's theorem since $f\in L^1_{loc}(D)$.
		Thus, $\langle\tilde{L}(P_D^*\lambda),\varphi\rangle=0$ for all $\varphi\in C_c^\infty(D)$.
	\end{proof}

	\begin{thm}\phantomsection\label{t:harmon L=0}
		Let $D$ be an open set and $u$ $L$-harmonic in $D$ with outer charge $\lambda$. Then $\tilde{L}(u,\lambda)=0$ in $D$.
	\end{thm}
	\begin{proof}
		Let $\varphi\in C_c^\infty(D)$. There is $U\subsub D$ with Lipschitz boundary such that $\supp \varphi\subset U$, i.e. $\varphi\in C_c^\infty(U)$. From \eqref{mean-value-property} for $u$ we have $u=P_U\tilde{\lambda}$ in $U$, where $\tilde\lambda(dy)=u(y)\mathbf{1}_{D\setminus U}(y)dy+\mathbf{1}_{D^c}(y)\lambda(dy)$. This means that $u$ is the Poisson integral on $U$ so by Lemma \ref{l:L(Plambda)=0} we have
		\begin{align*}
			\int\limits_DL\varphi(x) u(x)dx+\int\limits_{D^c}L\varphi(x)\lambda(dx)=\int\limits_U L\varphi(x)P_U\tilde{\lambda}(x)dx+\int\limits_{U^c}L\varphi(x)\tilde\lambda(dx)=0.
		\end{align*}
		Since $\varphi$ was arbitrary, we have the claim.
	\end{proof}
	\begin{rem}
		The proof of the previous theorem is valid in a much greater generality. Indeed, the only non-trivial part of the proof was the property $\tilde{L}(G_Df)=-f$ in $D$ proved in Proposition \ref{p:L(Gf)=-f}.  One can check that Proposition \ref{p:L(Gf)=-f} is true with the same proof for the isotropic unimodal L\'evy process with the condition \eqref{eq:scm_2.12} on the jumping kernel since the auxiliary results \cite[Lemma 3.5]{moritz} and  Lemma  \ref{r:konacnost od GDf} also hold in this setting. 
	\end{rem}
	We can extract a weakened converse claim of Theorem \ref{t:harmon L=0} using \cite[Lemma 3.3]{moritz}:
	\begin{thm}\phantomsection\label{ekvivalencija rjesenja}
		Let $D$ be an open set and $u\in\mathcal{L}^1$. If $\tilde{L}u=0$ in $D$, then $u$ has a modification that is $L$-harmonic in $D$.
	\end{thm}
	\begin{proof}
		If $\tilde{L}u=0$ in $D$, then it is proved in \cite[Lemma 3.3]{moritz} that for every Lipschitz $U\subset\subset D$ we have $u(\cdot)=P_Uu(\cdot)=\int_{U^c}u(y)P_U(\cdot,y)dy$ a.e.  in $U$. 
		
		Define  the  function $\tilde{u}:\R^d\to \R$ as $\tilde{u}=u$ on $D^c$ and for $x\in D$ choose some Lipschitz $U\subset\subset D$ such that $x\in U$ and define $\tilde{u}(x)=P_Uu(x)$. Let us show that $\tilde{u}$ is well defined. Suppose that we have Lipschitz sets $U_1\subset\subset D$ and $U_2\subset\subset D$ such that $x\in U_1\cap U_2$ and $P_{U_1}u(x)> P_{U_2}u(x)$. Since by Corollary \ref{c:P_D harm} $P_{U_j}u$ is continuous in $U_j$,  $j\in\{1,2\}$, there is $\varepsilon>0$ such that for every $y\in B(x,\varepsilon)\subset U_1\cap U_2$ we have $P_{U_1}u(y)> P_{U_2}u(y)+\varepsilon$. But  $u=P_{U_1}u=P_{U_2}u$ a.e. in $U_1\cap U_2$  so we have a contradiction. Hence, $\tilde{u}$ is well defined.
		
		 Recall that since $D$ is an open set, it is a countable union of balls. Also, every ball is a Lipschitz set so it is obvious from the construction of $\tilde u$ and the beginning of the proof that $u=\tilde{u}$ a.e. in $\R^d$.
		
		Now we prove that $\tilde u$ is harmonic in $D$. Note that since $u=\tilde u$ a.e., we have for all Lipschitz sets $V\subsub D$ and all $x\in V$
		\begin{align*}
			\tilde u(x)=\ex_x[u(X_{\tau_V})]=\int\limits_{V^c}u(y)P_V(x,y)dy=\int\limits_{V^c}\tilde u(y)P_V(x,y)dy=\ex_x[\tilde u(X_{\tau_V})].
		\end{align*}
		Let $x\in U\subset\subset D$ and take a Lipschitz set $V$ such that $U\subsub V\subsub D$. We have by the strong Markov property and the previous equality
		\begin{align*}
			\tilde{u}(x)&=\ex_x[\tilde u(X_{\tau_V})]=\ex_x\big[\ex_{X_{\tau_U}}[\tilde u(X_{\tau_V})]\big]=\ex_x[\tilde u(X_{\tau_U})].
		\end{align*} 
	\end{proof}

	\section{Accessible points and Martin kernel}\phantomsection\label{s:Martin kernel}
	In this section we give a summary of results concerning the Martin boundary. All of the results are already known but some are not plainly stated. Our goal is to state and prove results that are important for our article for the reader's convenience.
	
	In the case where only \ref{H1}  holds  many results concerning  the  Martin kernel can be proved only for bounded sets so the additional assumptions \ref{H2} and \ref{E2}  will be occasionally  assumed to get results for unbounded sets.
	
	For $D\subset \R^d$ let us denote
	\begin{align*}
		D^*\coloneqq\begin{cases}
			\overline{D},&\text{ if $D$ is  bounded},\\
			\overline{D}\cup\{\infty\},&\text{ if $D$ is  unbounded},
		\end{cases}\enskip \partial^* D\coloneqq\begin{cases}
			\partial{D},&\text{ if $D$ is  bounded},\\
			\partial{D}\cup\{\infty\},&\text{ if $D$ is  unbounded},
		\end{cases}
	\end{align*}
	where $\infty$ is an additional point  in  Alexandroff compactification and it is called \textit{the point at infinity}.
	\begin{defn}\phantomsection\label{accessible}
		 Let $D$ be an open set.  A point $y\in\partial D$ is called accessible from $D$ if
		\begin{align*}
			P_D(x_0,y)=\int_DG_D(x_0,z)j(|z-y|)dz=\infty,\quad \text{ for some $x_0\in D$}.
		\end{align*}
		The point at infinity is accessible from $D$ if
		\begin{align*}
			\mathbb{E}_{x_0}\tau_D=\int_DG_D(x_0,y)dy=\infty,\quad \text{ for some $x_0\in D$}.
		\end{align*}
		If $y\in\partial^* D$ is not accessible it is called inaccessible.
		The set of all accessible points is denoted by $\partial_MD$.
	\end{defn}

	\begin{rem}\phantomsection\label{r:acc}
		In \cite[Proposition 4.1 \& Remark 4.2]{vondra_RMI_2018}  the following claims were proved.
		\begin{enumerate}[(a)]
			\item
			Let $y\in\partial D$. If $P_D(x_0,y)<\infty$ for some $x_0\in D$, then $P_D(x,y)<\infty$ for all $x\in D$.
			\item
			Assume \ref{H2}. If $\mathbb{E}_{x_0}\tau_D<\infty$ for some $x_0\in D$, then $\mathbb{E}_{x}\tau_D<\infty$ for all $x\in D$.
		\end{enumerate} 
		Note that we could get the claim $(a)$ directly from Theorem \ref{t:thm_1_1} $(a)$. Also, from the definition of accessible points it is clear that if $\lambda$ is a signed measure on $D^c$ such that $P_D|\lambda|<\infty$, then $\lambda$ is concentrated on $\R^d\setminus(D\cup\partial_MD)$, i.e. $\lambda$ can have  no  mass on the set of  accessible  points.
	\end{rem}
	For an open $D\subset\R^d$ we fix an arbitrary point $x_0\in D$ and define the Martin kernel on $D$ by
	\begin{align}\phantomsection\label{d:Martin kernel}
		\begin{split}
			M_D(x,y)&\coloneqq\frac{G_D(x,y)}{G_D(x_0,y)},\quad x,y\in D,y\ne x_0,\\
			M_D(x,z_0)&\coloneqq \lim_{D\ni v\to z_0}\frac{G_D(x,v)}{G_D(x_0,v)},\quad x\in D,z_0\in\partial^* D.
		\end{split}
	\end{align}
	In \cite{vondra_fm_2016} and \cite{vondra_RMI_2018} many important and useful results about the Martin kernel of more general processes  than  the subordinate Brownian motion were proved. E.g. it was proved that $M_D(x,z_0)$ exists, is finite and strictly positive for every $z_0\in\partial^*D$ (with the additional assumptions \ref{H2} and \ref{E2}  if $z_0$ is the point at infinity). We summarize some of those results in the following theorem.
	\begin{thm}\phantomsection\label{Martinova jezgra}
		 Let $D$ be an open set, and  $z_0\in\partial^* D$.
		\begin{enumerate}[(a)]
			\item
			Let $z_0\in\partial_MD$ (for $z_0=\infty$ assume \ref{H2}). The function $x\mapsto M_D(x,z_0)$ is $L$-harmonic in $D$ with zero outer charge and for every open $U\subsub D$  it holds 
			\begin{align*}
				M_D(x,z_0)= \int_{D\setminus U} M_D(y,z_0)\omega_U^x(dy),\quad x\in U.
			\end{align*}
			\item
			Let $z_0\notin\partial_MD$ (for $z_0=\infty$ assume \ref{H2} and \ref{E2}). The function $x\mapsto M_D(x,z_0)$ is not $L$-harmonic in $D$ with zero outer charge and for every open $U\subsub D$  it holds 
			\begin{align*}
				M_D(x,z_0)> \int_{D\setminus U} M_D(y,z_0)\omega_U^x(dy),\quad x\in U.
			\end{align*}
		\end{enumerate}
	\end{thm}
	\begin{proof}
		First notice that by adding the assumptions \ref{H2} and \ref{E2} where needed all assumptions of claims from \cite{vondra_fm_2016} and \cite{vondra_RMI_2018} are satisfied, see \cite[Section 4.1]{vondra_fm_2016} and \cite[Section 4.1]{vondra_RMI_2018}. Furthermore, recall Lemma \ref{eq:comparability of j} for the assumption \textbf{E1} of \cite{vondra_fm_2016}.
		
		Suppose that $z_0\notin \partial_MD$. From \cite[Theorem 3.1]{vondra_fm_2016} we have that
		\begin{align*}
			M_D(x,z_0)=\begin{cases}
				\frac{P_D(x,z_0)}{P_D(x_0,z_0)}, &\text{ if $z_0\in\partial D$},\\
				\frac{\mathbb{E}_{x}\tau_D}{\mathbb{E}_{x_0}\tau_D}, &\text{ if $z_0=\infty$.}
			\end{cases}
		\end{align*}
		Hence for finite $z_0\notin\partial_MD$, $x\mapsto M_D(x,z_0)$ is $L$-harmonic with outer charge $\delta_{z_0}/P_D(x_0,z_0)$ but it is not $L$-harmonic with zero outer charge, see Corollary \ref{c:P_D harm}. Also, for every $x\in U\subset\subset D$ we have by the mean-value property of $L$-harmonic functions
		\begin{align*}
			M_D(x,z_0)&=\int_{D\setminus U} M_D(y,z_0)\omega_U^x(dy)+\frac{P_U(x,z_0)}{P_D(x_0,z_0)}\\
			&>\int_{D\setminus U} M_D(y,z_0)\omega_U^x(dy),\quad x\in U.
		\end{align*}
		If $z_0=\infty$, then $M_D(x,\infty)$ is not $L$-harmonic with zero outer charge because for $x\in U\subsub D$ we have
		\begin{align*}
			\int_{D\setminus U}M_D(y,\infty)\omega_U^x(dy)&=\frac{1}{\mathbb{E}_{x_0}\tau_D}\mathbb{E}_x\big[\mathbb{E}_{X_{\tau_U}}\tau_D\big]=\frac{1}{\mathbb{E}_{x_0}\tau_D}\mathbb{E}_x\left[\int_{\tau_U}^{\tau_D}\mathbf 1dt\right]\\
			&<\frac{\mathbb{E}_{x}\tau_D}{\mathbb{E}_{x_0}\tau_D}=M_D(x,\infty),
		\end{align*}
		where the  strict  inequality comes from the fact that for $x\in U$ there is $\varepsilon>0$ such that $B(x,\varepsilon)\subset U$ and $\mathbb{E}_x\tau_{U}\ge\mathbb{E}_x\tau_{B(x,\varepsilon)}>0$ by \cite[Lemma 4.3]{vondra_2012_SCM}.
		
		Suppose now that $z_0\in\partial_M D$. Then we have that $x\mapsto M_D(x,z_0)$ is $L$-harmonic with zero outer charge. For  the  finite point $z_0$ this follows from \cite[Theorem 1.2(b)]{vondra_fm_2016} (see the proof), or \cite[Theorem 1.1]{vondra_RMI_2018}, and for the point at infinity we apply \cite[Theorem 1.4(b)]{vondra_fm_2016}, or \cite[Theorem 1.3]{vondra_RMI_2018}. In either case by the mean-value property of $L$-harmonic functions we get for every $U\subsub D$ and all $x\in U$
		\begin{align*}
			M_D(x,z_0)&= \int_{D\setminus U} M_D(y,z_0)\omega_U^x(dy).
		\end{align*}
		
	\end{proof}
	\begin{rem}\phantomsection\label{o mvp za martinove}
		It will be very useful to note that in \cite{vondra_RMI_2018} two specific mean-value formulae were proved. If $z_0\in\partial_MD\setminus\{\infty\}$, then for every $r<\frac{1}{4}|z_0-x_0|$ and $U_r\coloneqq D\setminus\overline{B(z_0,r)}$
		\begin{align}\phantomsection\label{(3.14)}
			M_D(x,z_0)=\int_{U_r^c} M_D(y,z_0)\omega_{U_r}^x(dy),\quad x\in U_r,
		\end{align}
		see \cite[(3.14)]{vondra_RMI_2018}.
		
		Also,   if $z_0=\infty\in\partial_MD$  and we additionally assume \ref{H2}, then for every $R>4|x_0|$ and $U_R\coloneqq D\cap B(0,R)$
		\begin{align}\phantomsection\label{(3.4)}
			M_D(x,\infty)=\int_{U_R^c} M_D(y,\infty)\omega_{U_R}^x(dy),\quad x\in U_R,
		\end{align}
		see \cite[(3.4)]{vondra_RMI_2018}.
		
		In fact, from \eqref{(3.14)} it follows by using the strong Markov property that \eqref{(3.14)} is true for every $U\subset D$ open such that $z_0\notin\overline{U}$.  By similar reasoning \eqref{(3.4)} holds for every $U\subset D$ open and bounded such that $U_R\subset U$ for some $R>4|x_0|$.		
	\end{rem}

	\begin{defn}\phantomsection\label{def_1.13}
		Let $D\subset \R^d$ be an open set and $\mu$ a finite signed measure on $\partial^*D$ concentrated on $\partial_MD$. The Martin integral of $\mu$ is defined by
		\begin{align*}
			M_D\mu(x)&\coloneqq\int_{\partial_MD}M_D(x,y)\mu(dy),\quad x\in \R^d.
		\end{align*}
	\end{defn}

	\begin{rem}\phantomsection\label{r:komentar na definiciju od MDmu} 
		Let $\mu$ be a finite measure concentrated on $\partial_MD$. From $M_D(x_0,z)=1$, $z\in\partial^*D$, we see that $M_D\mu(x_0)=\mu(\partial_MD)$. It will follow from Corollary \ref{c:beskonacnost od M_Dmu} that $M_D\mu$ is finite at some point (or all points) if and only if $\mu$ is finite.	Also, due to harmonicity of $x\mapsto M_D(x,z_0)$ for $z_0\in\partial_MD$, it is easy to check that  $M_D\mu$ is $L$-harmonic in $D$ with outer charge zero. That is the reason why  we look,   regarding the Martin integral,  at  finite  measures concentrated on $\partial_MD$  in what follows.  
	\end{rem}

	\section{Boundary trace operator $W_D$ and representation of $L$-harmonic functions}\phantomsection\label{s:W_D and representation}
	
	 Let $D$ be an open set,  $u:D\to[-\infty,\infty]$, and  let $U\subsub D$ be a set with Lipschitz boundary   such that $x_0\in U$, where $x_0$ is the fixed point from the definition of the Martin kernel. We define  the signed  measure $\eta_Uu$ by
	\begin{align*}
		\eta_Uu(A)=\int\limits_A G_U(x_0,z)\left(\int\limits_{D\setminus U} j(|z-y|)u(y)dy\right)dz,\quad A\in\BB(\R^d).
	\end{align*}
	\begin{defn}\phantomsection\label{boundary operator}
		If  $(\eta_U|u|(D))_U$ is  bounded as  $U\uparrow D$ and $(\eta_Uu)_U$  weakly  converges   to a  signed  measure $\mu$ as $U\uparrow D$, then we denote $W_Du=\mu$, i.e. $W_Du\coloneqq\lim\limits_{U\uparrow D} \eta_Uu$.
	\end{defn}
	The boundary trace operator $W_D$ was used in \cite{bogdan_et_al_19} as  the  boundary condition in the Dirichlet problem for the fractional Laplacian and it was used as a tool to get  the  representation of non-negative $\alpha$-harmonic functions in \cite{bogdan_est_and_struct}.
	 As one can see, the definition of $W_D$ is rather delicate. It can be easily seen that for a bounded function $f$ we have $W_Df=0$ since $\eta_U|f|(D)\downarrow 0$. However, $W_D$ can be applied to many more functions, e.g. we will show that $W_D[M_D\mu]=\mu$ and $W_D[G_Df]=W_D[P_D\lambda]=0$, see also \cite[Proposition 4.8]{semilinear_bvw}. In what follows, we prove that some important properties of $W_D$ are also true in the case of subordinate Brownian motions and at the end of the article we will use the operator to get  the  representation of non-negative $L$-harmonic functions.

	\begin{lem}\phantomsection\label{l:W_D concentration}
		$W_Du$ is concentrated on $\partial^*D$.
	\end{lem}
	\begin{proof}
		Let $A\subsub D$. Then there is  a Lipschitz set $U_A\subsub D$  such that $x_0\in U_A$ and $A\subsub U_A$. Now we will show that $G_U(x_0,y)\asymp G_D(x_0,y)$, for all $y\in A$ and for all   Lipschitz $U$ such that $U_A\subsub U\subsub D$.
		
		Let $\varepsilon>0$ be such that $\overline{B(x_0,2\varepsilon)}\subset U_A$. For $y\in B(x_0,\varepsilon)$  and all Lipschitz $U$ such that $U_A\subsub U\subsub D$  we have
		\begin{align}\phantomsection\label{pomoc_ineq}
			G_{B(x_0,2\varepsilon)}(x_0,y)\le G_U(x_0,y)\le G_{\R^d}(x_0,y)\le C\, G_{B(x_0,2\varepsilon)}(x_0,y)
		\end{align}
		where $C>1$ is independent of $U$. Indeed, by \eqref{eq:G comparability} and \cite[Theorem 1.3]{grzywny_potential_kernels} we have for $y\in B(x_0,\varepsilon)$
		\begin{align}\phantomsection\label{Green kugle ocjene}
			\begin{split}
				G_{\R^d}(x_0,y)\le  c_1 \frac{1}{|x_0-y|^d\phi(|x_0-y|^{-2})},\\
				G_{B(x_0,2\varepsilon)}(x_0,y)\ge \frac{1}{ c_2}\frac{j(|x_0-y|)}{(K(|x_0-y|)+L(|x_0-y|))^2}.
			\end{split}	
		\end{align}
		where $K(r)=\int_{B(0,r)}\frac{|z|^2}{r^2}j(|z|)dz$ and $L(r)=\int_{B(0,r)^c}j(|z|)dz$. Define $h(r)=K(r)+L(r)=\int_{\R^d}\left(1\wedge \frac{|z|^2}{r^2}\right)j(|z|)dz$. By \cite[Eq. (6) and Lemma 1]{bogdan_density_and_tails_unimodal} we have that $h(r)\asymp \phi(\frac{1}{r^2})$ so by using \cite[Theorem 2.3]{vondra_2012_SCM} for all small enough $q>0$ we have that 
		\begin{align*}
			K(q)\le K(q)+L(q)= h(q) \le c_3 \phi(\frac{1}{q^2})\le  c_4 j(q)q^d.
		\end{align*}
		Using this inequality with inequalities \eqref{Green kugle ocjene} we get \eqref{pomoc_ineq}.
		
		For $y\in B(x_0,\varepsilon)^c\cap A$ notice that $0< c_5\le G_{U_A}(x_0,y)\le G_U(x_0,y)\le G_D(x_0,y)\le  c_6<\infty$ because Green functions are continuous and strictly positive on $B(x_0,r)^c\cap A$ since $A\subsub U_A\subsub D$. 
		Thus, $G_U(x_0,y)\asymp G_D(x_0,y)$, for all $y\in A$ and for all  Lipschitz $U$ such that $U_A\subsub U\subsub D$.
		
		Hence, for all such $U$ we have
		\begin{align*}
			\eta_U|u|(A)\asymp \int_A G_D(x_0,y)\underbrace{\int_{D\setminus U}j(|z-y|)|u(z)|dz}_{\text{$\downarrow 0$ as $U\uparrow D$}}dy\overset{U\uparrow D}{\longrightarrow} 0
		\end{align*}
		by the dominated convergence theorem.
	\end{proof}
	
	\begin{rem}\phantomsection\label{rubni operator}
		\begin{enumerate}[(a)]
			\item
			If we take a closer look at the proof of the previous lemma, we  have actually proved that if  $(\eta_U|u|(D))_U$ is bounded as $U\uparrow D$, then for every $A\subsub D$ we have
			\begin{align*}
				\lim_{U\uparrow D}\eta_U|u|(A)=0.
			\end{align*}
			\item
			The measures $(\eta_Uu)_U$  depend  on $x_0\in D$ but we can prove quite simply that for any other $x\in D$, the measures $$\eta^x_U|u|(dy)\coloneqq G_U(x,y)\left(\int_{D\setminus U}j(|z-y|)|u(z)|dz\right) dy$$ are also bounded as $U\uparrow D$ if $(\eta_U|u|)_U$ are.	Indeed, let $M\coloneqq \limsup\limits_{U\uparrow D}\eta_U|u|(D)$. Notice that by Fubini's theorem
			\begin{align*}
				\eta_U|u|(D)&=\int_D G_U(x_0,z)\left(\int_{D\setminus U} j(|z-y|)|u(y)|dy\right)dz\\
				&=\int_{D\setminus U}P_U(x_0,y)|u(y)|dy.
			\end{align*}
			Find $R\in (0,1)$ such that $\delta_D(x_0)>2R$ and let $(U_n)_n$ be some increasing sequence of Lipschitz  sets such that $x_0\in U_1$, $\delta_{U_1}(x_0)>R$, and such that for all $n\in\N$ it holds $U_n\subsub D$ and $\cup_n U_n=D$.  Also, fix some $\tilde{y}\in \overline{D}^c$. Theorem \ref{t:thm_1_1} yields that there is $C>0$ such that for all $n\in\N$, all $x\in B(x_0,R/2)$, and all $y\in U_n^c$
			\begin{align*}
				P_{U_n}(x,y)&\le C \frac{P_{U_n}(x,\tilde y)}{P_{U_n}(x_0,\tilde{y})}P_{U_n}(x_0,y).
			\end{align*}
			Notice that
			\begin{align*}
				\frac{P_{U_n}(x,\tilde y)}{P_{U_n}(x_0,\tilde{y})}\le \frac{P_{D}(x,\tilde y)}{P_{U_1}(x_0,\tilde{y})}\le \frac{\max_{z\in B(x_0,R/2)}P_{D}(z,\tilde y)}{P_{B(x_0,R/2)}(x_0,\tilde{y})}\le c_1<\infty,
			\end{align*}
			where $c_1>0$ depends on $x_0$, $R$ and $\tilde{y}$  but it is independent of $n\in\N$ and $x\in B(x_0,R/2)$. Finiteness  of $c_1$  is due to  the  continuity of the Poisson kernel. Thus, there is $c_2>0$ such that for all $n\in\N$, all $x\in B(x_0,R/2)$ and all $y\in U_n^c$ we have
			$P_{U_n}(x,y)\le c_2 P_{U_n}(x_0,y)$. Hence
			\begin{align*}
				\eta^x_{U_n}|u|(D)&=\int_{D\setminus U_n}P_{U_n}(x,y)|u(y)|dy\\
				&\le c_2 \int_{D\setminus U_n}P_{U_n}(x_0,y)|u(y)|dy\le c_2 \cdot M,
			\end{align*}
			i.e. $(\eta^x_{U}|u|(D))_U$ is bounded as $U\uparrow D$, for all $x\in D$.
		\end{enumerate}
	\end{rem}

	\begin{prop}\phantomsection\label{lemma 1.17_bodi et ali}
		Let  $D$ be an open set, $f:D\to[-\infty,\infty]$ such that  $G_D|f|(x)<\infty$ for some $x\in D$, and $\lambda$ a $\sigma$-finite signed measure on $D^c$ such that \eqref{eq:konacnost tocke PDlambda} holds.  Then
		\begin{align*}
			W_D[G_Df]=W_D[P_D\lambda]=0.
		\end{align*}
	\end{prop}
	\begin{proof}
		The proof is the same as in the isotropic $\alpha$-stable case, see \cite[Lemma 1.17]{bogdan_et_al_19}.
	\end{proof}

	 We now focus on proving the mentioned property $W_D[M_D\mu]=\mu$. We use an adaptation of the technique used in \cite{bogdan_est_and_struct} where the property was shown for the isotropic $\alpha$-stable process. In the next few results we have twofold statements - for sets near the origin, and for sets away of the origin.  In the isotropic $\alpha$-stable case the Kelvin transform allowed the authors to deal only with sets near the origin but in our setting this is not the case.
	
	Let us recall the definition of  the  relative oscillation of a positive function $f$  on a nonempty set $D$
	$$ \mathrm{RO}_Df\coloneqq\frac{\sup_{x\in D}f(x)}{\inf_{x\in D}f(x)}.$$
	If $D=\emptyset$ we put $\mathrm{RO}_Df=1$.
	
	The first lemma is the one that generalizes \cite[Lemma 8]{bogdan_est_and_struct}.

	\begin{lem}\phantomsection\label{lemma 8}
		\begin{enumerate}[(a)]
			\item
			 For every $R\in(0,1)$ and $\eta>0$ there exists $\delta>0$ such that for all open $D\subset B_R$  and all $\sigma$-finite measures $\lambda_1$, $\lambda_2$ on $B_R^c$ satisfying \eqref{eq:konacnost tocke PDlambda} we have
			\begin{align}\phantomsection\label{47_18}
				\mathrm{RO}_{D\cap B_{\delta}}\frac{P_D\lambda_1}{P_D\lambda_2}\le 1+\eta.
			\end{align}
			\item
			Assume \ref{H2} and \ref{E2}.  For every $R\ge1$ and $\eta>0$ there exists $\delta>0$ such that for  all open $D\subset \overline{B}_{R}^c$  and all $\sigma$-finite measures $\lambda_1$, $\lambda_2$ on $\overline{B}_R$ satisfying \eqref{eq:konacnost tocke PDlambda} we have
			\begin{align}\phantomsection\label{47_18 - infty}
				\mathrm{RO}_{D\cap \overline B^c_{ 1/\delta}}\frac{P_D\lambda_1}{P_D\lambda_2}\le 1+\eta.
			\end{align}
		\end{enumerate}
	\end{lem}
	
	Before we bring the proof let us emphasize the results of the previous lemma. In both parts of the lemma  $\delta$ is chosen independently of the set $D$, and  the  measures $\lambda_1$ and $\lambda_2$. In similar results on the relative oscillation of harmonic functions, e.g. \cite[Proposition 2.5, Proposition 2.11]{vondra_RMI_2018}, $\delta$ is dependent on the set $D$, see also the proofs of \cite[Theorem 2.4, Theorem 2.8]{vondra_fm_2016}. This subtle but big difference will be used as a crucial and indispensable step in proving $W_D[M_D\mu]=\mu$, see \eqref{eq:crucial step}.
	
	Moreover, the previous lemma yields that the Martin kernel $M_D(x,z)$ is well defined and strictly positive for $x\in D$ and $z\in\partial^*D$. To this end, recall $M_D(x,z)=\lim_{y\to z}\frac{G_D(x,y)}{G_D(x_0,y)}$, where if $z=\infty$ we look at the limit as $|y|\to\infty$. Since the process $X$ is translation invariant, we can assume that for the finite point $z$ it holds $z=0$. Further, notice that from \eqref{eq:GD SMP} we have for $\rho>0$
	$$G_D(\tilde x,y)=P_{D\cap B_\rho}[G_D(\tilde x,v)dv](y),\quad \tilde x\in D\setminus \overline{B}_\rho,\,y\in D\cap B_\rho,$$ and
	$$G_D(\tilde x,y)=P_{D\cap \overline B^c_\rho}[G_D(\tilde x,v)dv](y),\quad\tilde x\in D\cap {B}_\rho,\,y\in B\setminus \overline B_\rho.$$ Now the claim follows from \eqref{47_18}  and \eqref{47_18 - infty}. However, for this result  the uniformity of $\delta$ was not important.
	
	\begin{proof}[Proof of Lemma \ref{lemma 8}]
		We prove only part $(b)$. The proof of part $(a)$ is almost identical to the proof of the \cite[Lemma 8]{bogdan_est_and_struct}. The only difference is that instead of the unit ball $B$ we look at the ball $B_R$ and instead of \cite[Eq. (48)]{bogdan_est_and_struct} we use Lemma \ref{eq:comparability of j}. The proof of part $(b)$ follows the same idea and we present the proof to emphasize the differences. To establish a connection between our proof and the proof of \cite[Lemma 8]{bogdan_est_and_struct} we will keep a similar notation.
		
		 For an open set $D$ and $R,p,q>0$  denote by
		\begin{align*}
			&D_p=D\cap \overline{B}^c_p,\\
			&D_p^R=(D\setminus D_p)\cup \overline B_R,\\
			&D_{p,q}=D_{q}\setminus D_{p}.
		\end{align*}
		For  a measure $\mu$ let 
		\begin{align*}
			\Lambda_{0,p}(\mu)&=\int_{\overline B_{p}}\mu(dy),\\
			\Lambda_{0,p,q}(\mu)&=\int_{D_{p,q}}\mu(dy).
		\end{align*}
		
		Fix $R\ge 1$,  $D\subset \overline B^c_R$, and $\sigma$-finite measures $\lambda_1$ and $\lambda_2$ on $\overline B_R$ satisfying \eqref{eq:konacnost tocke PDlambda}. We will see at the end of the proof that $\delta$ will not depend on $D$, $\lambda_1$ or $\lambda_2$, so this is not a loss of generality. Let $c$ denote $C(\phi)>1$ of Lemma \ref{lemma 7}$(b)$ and notice that Theorem \ref{t:thm_1_1}$(b)$ holds with the constant $C=c^4$. Thus, \eqref{47_18 - infty} holds for $\delta=\frac12$ with $1+\eta$ replaced by $c^4$.
		We denote
		\begin{align*}
			f_i&=P_D\lambda_i,&f_i^{pR,qR}=P_{D_{pR}}[\mathbf{1}_{D_{pR,qR}} P_D^*\lambda_i],\enskip&\widetilde{f}_i^{pR,qR}=P_{D_{pR}}[\mathbf{1}_{D_{qR}^R} P_D^*\lambda_i],\\
			f_i^*&=P_D^*\lambda_i,&f_i^{pR,qR*}=P_{D_{pR}}^*[\mathbf{1}_{D_{pR,qR}} P_D^*\lambda_i],\enskip&\widetilde{f}_i^{pR,qR*}=P_{D_{pR}}^*[\mathbf{1}_{D_{qR}^R} P_D^*\lambda_i].
		\end{align*}
		Recall that $P_D\lambda$ satisfies the mean-value formula for every $U\subset D$ by Corollary \ref{c:P_D harm}. Hence, using \eqref{eq:PD Lipschitz set} we have $f_i=f_i^{pR,qR}+\widetilde{f}_i^{pR,qR}$ and $f_i^*=f_i^{pR,qR*}+\widetilde{f}_i^{pR,qR*}$, for $i=1,2$. For $\delta\in(0,\frac12]$ we denote $m_{R/\delta}=\inf_{D_{R/\delta}} (f_1/f_2)$ and $M_{R/\delta }=\sup_{D_{R/\delta}} (f_1/f_2)$. As we have already noted we have $M_{R/\delta}\le c^4 m_{R/\delta}$.
		
		Let $\varepsilon>0$ such that $1+\varepsilon<c$ and let $q\ge 2$. Assumption \ref{E2} yields that there is $p=p(q,\varepsilon, R)> 2q$ such that for $z\in D_{pR/2}$ and $y\in \overline B_{qR}$ we have
		\begin{align}\phantomsection\label{nastelavanje2}
			\frac{1}{1+\varepsilon}j(|z|)\le j(|z-y|)\le (1+\varepsilon)j(|z|).
		\end{align}
		Thus, for $x\in D_{pR/2}$ we have
		\begin{align*}
			\widetilde{f}_i^{pR/2,qR}(x)&=\int_{D_{qR}^R}\int_{D_{pR/2}} G_{D_{pR/2}}(x,z)j(|z-y|)dzf_i^*(dy)\\
			&\le (1+\varepsilon) \Lambda_{0,qR}(f_i^*) P_{D_{pR/2}}(x,0),
		\end{align*}
		and similarly
		\begin{align*}
			\widetilde{f}_i^{pR/2,qR}(x)\ge (1+\varepsilon)^{-1} \Lambda_{0,qR}(f_i^*) P_{D_{pR/2}}(x,0).
		\end{align*}
		
		Let us examine consequences of the following assumption:
		\begin{align}\phantomsection\label{49_18 - infty}
			\Lambda_{0,pR,qR}(f_i^*)\le \varepsilon \Lambda_{0,qR}(f_i^*),\quad i=1,2.
		\end{align}
		If \eqref{49_18 - infty} is true, then using Lemma \ref{lemma 7}$(b)$ we have for $x\in D_{pR}$
		\begin{align*}
			f_i^{pR/2,qR}(x)&\le c P_{D_{pR/2}}(x,0)\Lambda_{0,pR}(f_i^{pR/2,qR*})\le c P_{D_{pR/2}}(x,0)\Lambda_{0,pR,qR}(f_i^{*})\\
			&\le c \varepsilon P_{D_{pR/2}}(x,0) \Lambda_{0,qR}(f_i^*).
		\end{align*}
		Recall that $f_i=f_i^{pR/2,qR}+\widetilde{f}_i^{pR/2,qR}$ so if \eqref{49_18 - infty} holds, we have for $x\in D_{pR}$
		\begin{align}\phantomsection\label{50_18 - infty}
			\frac{(1+\varepsilon)^{-1}\Lambda_{0,qR}(f_1^*)}{(c\varepsilon+1+\varepsilon)\Lambda_{0,qR}(f_2^*)}\le\frac{f_1(x)}{f_2(x)}\le \frac{(c\varepsilon+1+\varepsilon)\Lambda_{0,qR}(f_1^*)}{(1+\varepsilon)^{-1}\Lambda_{0,qR}(f_2^*)}
		\end{align}
		and finally
		\begin{align}\phantomsection\label{51_18 - infty}
			\mathrm{RO}_{D_{pR}} \frac{f_1}{f_2}\le (c\varepsilon+1+\varepsilon)^2(1+\varepsilon)^2.
		\end{align}
		We are satisfied with \eqref{51_18 - infty} for now.
		
		Let $2\le \bar{q}<\bar{p}/4<\infty$, $g=f_1^{\bar{p}R/2,\bar{q}R}-m_{\bar{q}R}f_2^{\bar{p}R/2,\bar{q}R}$, and $h=M_{\bar{q}R}f_2^{\bar{p}R/2,\bar{q}R}-f_1^{\bar{p}R/2,\bar{q}R}$. Note that on $D_{\bar{p}R/2}$  the functions  $g$ and $h$ are  the  Poisson integrals of non-negative measures. If $D_{\bar{p}R}\ne\emptyset$, then by \eqref{44_18 - infty}
		\begin{align*}
			\sup_{D_{\bar{p}R}}\frac{f_1^{\bar{p}R/2,\bar{q}R}}{f_2^{\bar{p}R/2,\bar{q}R}}-m_{\bar{q}R}=\sup_{D_{\bar{p}R}}\frac{g}{f_2^{\bar{p}R/2,\bar{q}R}}&\le c^4 \inf_{D_{\bar{p}R}}\frac{g}{f_2^{\bar{p}R/2,\bar{q}R}}\\
			&=c^4\left(\inf_{D_{\bar{p}R}}\frac{f_1^{\bar{p}R/2,\bar{q}R}}{f_2^{\bar{p}R/2,\bar{q}R}}-m_{\bar{q}R}\right),
		\end{align*}
		and similarly
		\begin{align*}
			M_{\bar{q}R}-\inf_{D_{\bar{p}R}}\frac{f_1^{\bar{p}R/2,\bar{q}R}}{f_2^{\bar{p}R/2,\bar{q}R}}	\le c^4\left(M_{\bar{q}R}-\sup_{D_{\bar{p}R}}\frac{f_1^{\bar{p}R/2,\bar{q}R}}{f_2^{\bar{p}R/2,\bar{q}R}}\right).
		\end{align*}
		By adding these two inequalities we obtain
		\begin{align}\phantomsection\label{52_18 - infty}
			(c^4+1)\left(\sup_{D_{\bar{p}R}}\frac{f_1^{\bar{p}R/2,\bar{q}R}}{f_2^{\bar{p}R/2,\bar{q}R}}-\inf_{D_{\bar{p}R}}\frac{f_1^{\bar{p}R/2,\bar{q}R}}{f_2^{\bar{p}R/2,\bar{q}R}}\right)\le(c^4-1)(M_{\bar{q}R}-m_{\bar{q}R}).
		\end{align}
		
		Let us examine consequences of the following assumption:
		\begin{align}\phantomsection\label{53_18 - infty}
			\Lambda_{0,\bar{q}R}(f_i^*)\le \varepsilon \Lambda_{0,\bar pR/2,\bar qR}(f_i^*),
		\end{align}
		for $\bar p$ big enough such that $j(|z-y|)\le c j(|z|)$ for all $z\in D_{\bar pR/2}$ and $y\in \overline B_{\bar qR}$ (see \eqref{nastelavanje2}). We have for all $x\in D_{\bar pR/2}$ and $y\in \overline B_{\bar qR}$
		\begin{align*}
			P_{D_{\bar pR/2}}(x,y)=\int_{D_{\bar pR/2}}G_{D_{\bar pR/2}}(x,z)j(|z-y|)dz\le cP_{D_{\bar pR/2}}(x,0),
		\end{align*}
		hence
		\begin{align*}
			\widetilde{f}_i^{\bar{p}R/2,\bar{q}R}(x)=\int_{D_{\bar qR}^R}P_{D_{\bar pR/2}}(x,y)f_i^*(dy)&\le cP_{D_{\bar pR/2}}(x,0)\Lambda_{0,\bar qR}(f_i^*).
		\end{align*}
		From the previous inequality using the assumption \eqref{53_18 - infty} and Lemma \ref{lemma 7}$(b)$ we have for $x\in D_{\bar p R}$
		\begin{align*}
			\widetilde{f}_i^{\bar{p}R/2,\bar{q}R}(x)&\le c\,\varepsilon P_{D_{\bar pR/2}}(x,0)\Lambda_{0,\bar pR/2,\bar qR}(f_i^*)\\
			&\le c\,\varepsilon P_{D_{\bar pR/2}}(x,0)\Lambda_{0,\bar pR}(f_i^{\bar{p}R/2,\bar{q}R*})\le c^2\varepsilon f_i^{\bar{p}R/2,\bar{q}R}(x).
		\end{align*}
		Recall $f_i=f_i^{\bar{p}R/2,\bar{q}R}+\widetilde{f}_i^{\bar{p}R/2,\bar{q}R}$ on $D_{\bar p R/2}$ so the previous inequality and \eqref{52_18 - infty} yield
		\begin{align*}
			(c^4+1)\left(	M_{\bar{p}R}/(1+c^2\varepsilon)-m_{\bar{p}R}(1+c^2\varepsilon)\right)\le(c^4-1)(M_{\bar{q}R}-m_{\bar{q}R}).
		\end{align*}
		Since $m_{\bar{p}R}\ge m_{\bar{q}R}$, dividing by $m_{\bar{q}R}$ we finally get
		\begin{align}\phantomsection\label{54_18 - infty}
			\mathrm{RO}_{D_{\bar pR}}\frac{f_1}{f_2}\le (1+c^2\varepsilon)^2+(1+c^2\varepsilon)\frac{c^4-1}{c^4+1}\left(\mathrm{RO}_{D_{\bar qR}}\frac{f_1}{f_2}-1\right).
		\end{align}

		We now come to the conclusion of our considerations. Let $\eta > 0$. If $\varepsilon$ is small
		enough, then the right hand side of \eqref{51_18 - infty} is smaller than $1+\eta$ and the right hand side
		of \eqref{54_18 - infty} does not exceed $\varphi(\mathrm{RO}_{D_{\bar qR}} (f_1/f_2))$, where
		\begin{align*}
			\varphi(t)=1+\frac{\eta}{2}+\frac{c^4}{c^4+1}(t-1),\quad t\ge 1.
		\end{align*}
		Let $\varphi^1=\varphi$, $\varphi^{l+1}=\varphi\circ\varphi^l$, $l\in\N$. Observe that $\varphi$ is an increasing linear contraction with a fixed point $t=1+\eta(c^4+1)/2$. Thus the $l$-fold compositions $\varphi^l(c^4)$ converge to $1+\eta(c^4+1)/2$ as $l\to\infty$. In what follows let $l$ be such that 
		\begin{align*}
			\varphi^l(c^4)<1+\eta(c^4+1).
		\end{align*}
		Let $k$ be the smallest integer such that $k-1>c^2/\varepsilon^2$. We denote $n=lk$.  Note that $n$ depends only on $\eta$ and $\phi$.  Let $q_0=2$, $q_{j+1}=p(q_j, \varepsilon,R)$ for $j=0,1,\dots,n-1$, from \eqref{nastelavanje2}, and $\delta=\frac{1}{q_n}$. Note that $\delta$ depends only  on  $\eta$, $R$ and $\phi$. If for any $j<n$, \eqref{49_18 - infty} holds with $q=q_j$ and $p=p(q)=q_{j+1}$, then
		\begin{align*}
			\mathrm{RO}_{D_{R/\delta }}\frac{f_1}{f_2}\le\mathrm{RO}_{D_{q_{j+1}R}}\frac{f_1}{f_2}\le 1+\eta,
		\end{align*}
		by the definition of $\varepsilon$ and \eqref{51_18 - infty}. Otherwise for $j=0,\dots,n-1$, we have $\Lambda_{0,q_{j+1}R,q_{j}R}(f_i^*)>\varepsilon \Lambda_{0,q_{j}R}(f_i^*)$ for $i=1$ or $i=2$. Note that by Lemma \ref{lemma 7}$(b)$
		\begin{align*}
			c^{-1}\frac{f_i(x)}{\Lambda_{0,q_{j}R}(f_i^*)}\le P_{D_{q_jR/2}}(x,0)\le c\frac{f_{3-i}(x)}{\Lambda_{0,q_{j}R}(f_{3-i}^*)},\quad x\in D_{q_{j+1}R,q_{j}R}.
		\end{align*}
		Hence $\Lambda_{0,q_{j+1}R,q_{j}R}(f_i^*)/\Lambda_{0,q_{j}R}(f_i^*)\le c^2 \Lambda_{0,q_{j+1}R,q_{j}R}(f_{3-i}^*)/\Lambda_{0,q_{j}R}(f_{3-i}^*)$ and so $\Lambda_{0,q_{j+1}R,q_{j}R}(f_i^*)\ge c^{-2}\varepsilon\Lambda_{0,q_{j}R}(f_i^*)$ for both $i=1$ and $i=2$ (and all $j=0,\dots,n-1)$. If $0\le j<l$ and $\bar p=q_{(j+1)k}$, $\bar q=q_{jk}$, then
		\begin{align*}
			\Lambda_{0,\bar pR/2,\bar qR}(f_i^*)\ge \Lambda_{0,q_{(j+1)k-1}R,q_{jk}R}(f_i^*)\ge(k-1)\frac{\varepsilon}{c^{2}}\Lambda_{0,\bar qR}(f_i^*)\ge \varepsilon^{-1}\Lambda_{0,\bar qR}(f_i^*),
		\end{align*}
		so that \eqref{53_18 - infty} is satisfied. We conclude that \eqref{54_18 - infty} holds. Recall that $q_0=2$ and $\mathrm{RO}_{D_{2R}}(f_1/f_2)\le c^4$. By the definition of $l$ and  the  monotonicity of $\varphi$
		\begin{align*}
			\mathrm{RO}_{D_{q_{lk}R}}\frac{f_1}{f_2}\le\varphi\left(\mathrm{RO}_{D_{q_{(l-1)k}R}}\frac{f_1}{f_2}\right)\le \dots\le \varphi^l\left(\mathrm{RO}_{D_{q_{0}R}}\frac{f_1}{f_2}\right)\le 1+\eta(c^4+1),
		\end{align*}
		i.e. $\mathrm{RO}_{D_{R\delta}}\frac{f_1}{f_2}\le 1+\eta(c^4+1)$. Since $\eta>0$ was arbitrary and $\delta$ is dependant only on $\eta$, $R$ and $\phi$, the proof is complete.
	\end{proof}

	\begin{cor}\phantomsection\label{c:omjer reg harmonijskih}
		Let $D$ be an open set, $D_{reg}$ the set of all regular points for $D$, $z\in\partial D$, and $0<r<1\le R$.
		\begin{enumerate}[(a)]
			\item 
			Let $f_1$ and $f_2$ be non-negative functions which are regular harmonic in $D\cap B(z,r)$ and $f_i=0$ on $(\overline D^c\cup D_{reg})\cap B(z,r)$, $i=1,2$. Then
			$$ \lim_{D\ni x\to z}\frac{f_1(x)}{f_2(x)}$$
			exists and is finite.
			\item
			 Assume \ref{H2} and \ref{E2}.  If $f_1$ and $f_2$ are non-negative functions which are regular harmonic in $D\cap \overline{B}^c_R$ and $f_i=0$ on $(\overline D^c\cup D_{reg})\cap \overline B_R^c$, $i=1,2$, then
			$$ \lim_{D\ni x\to\infty}\frac{f_1(x)}{f_2(x)}$$
			exists and is finite.
		\end{enumerate}
		Moreover, the speed of convergence in the limits above does not depend on the set $D$.
	\end{cor}
	The previous corollary is an immediate consequence of Lemma \ref{lemma 8}, cf. \cite[Theorem 2.4, Theorem 2.8]{vondra_fm_2016} and  \cite[Corollary 2.6, Corollary 2.12]{vondra_RMI_2018} where the speed of convergence depends on the set $D$.
	\begin{proof}[Proof of Corollary \ref{c:omjer reg harmonijskih}]
		For part $(a)$ it is enough to notice that from the assumptions of the corollary we have for $x\in D\cap B(z,r)$ and both $i=1,2$
		\begin{align*}
			f_i(x)=\int_{D^c\cup B(z,r)^c}f_i(y)\omega^x_{D\cap B(z,r)}(dy)=\int_{B(z,r)^c}P_{D\cap B(z,r)}(x,y)f_i(y)dy.
		\end{align*}
		The claim now follows from Lemma \ref{lemma 8}$(a)$.  Part $(b)$ follows similarly. 
	\end{proof}

	The following results generalize \cite[Lemma 12]{bogdan_est_and_struct}.
	\begin{lem}\phantomsection\label{lemma 12}
		For every $0<\rho<1$ and $\eta>0$ there is $r>0$ such that for all open $D$  it holds 
		\begin{align}\phantomsection\label{85_18}
			\mathrm{RO}_{y\in \overline{D}\cap B_r} M_D(x,y)\le 1+\eta,\quad \text{if } x,x_0\in D\setminus \overline{B}_\rho,
		\end{align}
		and with the additional assumptions \ref{H2} and \ref{E2}  it holds 
		\begin{align}\phantomsection\label{86_18}
			\mathrm{RO}_{y\in D^*\setminus \overline B_{1/r}} M_D(x,y)\le 1+\eta,\quad \text{if } x,x_0\in D \cap B_{1/\rho}.
		\end{align}
	\end{lem}
	\begin{proof}
		Let $1>\rho>r>0$. Note that $$\sup\limits_{y\in \overline{D}\cap B_r} M_D(x,y)=\sup\limits_{y\in D\cap B_r}\frac{G_D(x,y)}{G_D(x_0,y)},\enskip\inf\limits_{y\in \overline{D}\cap B_r} M_D(x,y)=\inf\limits_{y\in D\cap B_r}\frac{G_D(x,y)}{G_D(x_0,y)}.$$
		Since $G_D(\tilde x,y)=P_{D\cap B_\rho}[G_D(\tilde x,v)dv](y)$ for $\tilde x\in D\setminus \overline{B}_\rho$, the claim $(a)$ follows from Lemma \ref{lemma 8}$(a)$. For part $(b)$ we apply Lemma \ref{lemma 8}$(b)$ in a similar way.
	\end{proof}
	
	\begin{rem}\phantomsection\label{nepre. Martin}
		From the previous lemma it is clear that the function $z\mapsto M_D(x,z)$ is continuous for every $x\in D$.
	\end{rem}
	
	Now we state two lemmas that appeared in \cite{bogdan_est_and_struct} for the case of the isotropic $\alpha$-stable process. The lemmas will be useful for proving uniqueness of representation of non-negative $L$-harmonic functions with zero outer charge.
	
	\begin{lem}\phantomsection\label{lemma 9}
		 Let $D$ be an open set.  Suppose that $0 \le  g \le  f$ on $D$, and that $f$, $g$ are $L$-harmonic in $D$ with
		zero outer charge. If $U\subset D$ and $f(x) =\int_{U^c} f (y)\omega^x_U(dy)$, $x\in U$, then $g(x) =\int_{U^c} g (y)\omega^x_U(dy)$, $x\in U$.
	\end{lem}
	\begin{proof}
		The proof is the same as in \cite[Lemma 9]{bogdan_est_and_struct}.
	\end{proof}

	\begin{lem}\phantomsection\label{lemma 10}
		Let $D_1$ and $D_2$ be open sets such that
		\begin{align*}
			\mathrm{dist}(D_1 \setminus D_2 , D_2 \setminus D_1) > 0 .
		\end{align*}
		Set $D = D_1\cup D_2$ and assume that $\omega_D^x(D^c) > 0$ for one (and therefore for all)
		$x\in D$. Let $f \ge 0$ be a function on $\R^d$ such that $f = 0$ on $D^c$, and for $i = 1, 2$ and all $x\in D_i$ we have
		\begin{align*}
			f (x) =\int	f (y)\omega_{D_i}^x (dy).
		\end{align*}
		Let $D_1$ be bounded and if $D_2$ is unbounded assume \ref{H2}. Then $f = 0$ on the whole of $D$.
	\end{lem}
	\begin{proof}
		The proof is the same as in \cite[Lemma 10]{bogdan_est_and_struct} where for inequalities $(70)$ and $(71)$ we use the Harnack inequality for the subordinate Brownian motion \cite[Theorem 7]{grzywny_harnack}.
	\end{proof}

	Now we have a generalization of \cite[Lemma 14]{bogdan_est_and_struct}.
	\begin{prop}[Martin representation]\phantomsection\label{lemma 14}
		Let $D$ be an open set. If $D$ is unbounded we  additionally assume \ref{H2} and \ref{E2}. Suppose $f\ge 0$ is $L$-harmonic on $D$ with zero outer charge. Then there is a unique finite measure $\mu\ge0$ on $\partial_M D$ such that
		\begin{align}\phantomsection\label{87_18}
			f(x)=\int_{\partial_M D}M_D(x,y)\mu(dy),
		\end{align}
		and we have $W_Df=\mu$. Conversely, if  $\mu$ is a finite measure on $\partial_M D$ and $f(x)\coloneqq\int_{\partial_M D}M_D(x,y)\mu(dy)$, then $f$ is $L$-harmonic with zero outer charge.
	\end{prop}
	
	Before we prove the proposition we connect the result with the Martin boundary of $D$ with respect to $X^D$ in the sense of Kunita-Watanabe, see \cite{kunita_watanabe}.  From \cite{vondra_fm_2016,vondra_RMI_2018} it follows that in our setting the (abstract) Martin boundary of the set $D$ can be identified with $\partial^*D$. Also, the \textit{minimal} Martin boundary can be identified with $\partial_MD$. However, in \cite[Corollary 1.2 \& Corollary 1.4]{vondra_RMI_2018} the Martin representation of harmonic functions with respect to $X^D$ was proved only for the case $\partial_MD=\partial^*D$, cf. \cite[Theorem 4]{kunita_watanabe}. Hence, the proposition above extends \cite[Corollary 1.2 \& Corollary 1.4]{vondra_RMI_2018} on more general sets but on less general processes.
	\begin{proof}[Proof of Proposition \ref{lemma 14}]
		The second claim is almost trivial. Since  $\mu$ is a finite measure on $\partial_M D$, we have that $f\coloneqq M_D\mu$ is $L$-harmonic in $D$ with zero outer charge because of the harmonicity of the Martin kernel.
		
		The first claim is proved similarly as in \cite[Lemma 14]{bogdan_est_and_struct} but because of some differences at the end of the proof we give the full proof for the reader's convenience. Let $(D_n)_n$ denote an increasing sequence of open sets   with Lipschitz boundary such that for all $n\in \N$ we have $D_n\subsub D$ and $D=\bigcup_{n=1}^\infty D_n$. By the mean-value property  we have for $x\in D_n$
		\begin{align*}
			f(x)&=\int\limits_{D\setminus D_n}P_{D_n}(x,y)f(y)dy\\
			&=\int\limits_{D_n}M_{D_n}(x,v)\left(G_{D_n}(x_0,v)\int\limits_{D\setminus D_n}j(|v-y|)f(y)dy\right)dv\\
			&=\int\limits_{D_n}M_{D_n}(x,v)\eta_{D_n}f(dv),
		\end{align*}
		where $\eta_{D_n}f$ is the measure from Definition \ref{boundary operator}. For brevity's sake, we write $\eta_n$ for $\eta_{D_n}f$. Since $\eta_n(D)=f(x_0)<\infty$, by considering a subsequence we may assume that the sequence $(\eta_n)_n$ weakly converges on $D^*$ to a finite non-negative measure $\mu^*$.
		It follows from Lemma \ref{l:W_D concentration},  more precisely  Remark \ref{rubni operator}$(a)$, that $\mu^*$ is supported on $\partial^*D$.
		
		Let $\varepsilon>0$ and $x\in D$. By Lemma \ref{lemma 12} for every $y\in\partial^*D$ there exists a neighbourhood $V_y$ of $y$ such that 
		\begin{align}\phantomsection\label{eq:crucial step}
			\mathrm{RO}_{V_y\cap U^*}M_U(x,\cdot)\le 1+\varepsilon,
		\end{align} 
		for all $U\in\{D,D_1,D_2,\dots\}$. From $\{V_y:y\in\partial^*D\}$, we select a finite family $\{V_j:j=1,\dots,m\}$ such that $V\coloneqq V_1\cup\dots\cup V_m\supset \partial^*D$. For $j\in\{1,\dots,m\}$ let $z_j\in D\cap V_j$. Let $k$ be so large that for $n>k$ we have $z_j\in D_n$ and 
		\begin{align*}
			(1+\varepsilon)^{-1}\le \frac{M_D(x,z_j)}{M_{D_n}(x,z_j)}\le (1+\varepsilon),\quad j=1,\dots,m.
		\end{align*}
		The last inequality can be achieved because $G_{D_n}\uparrow G_D$ pointwise  in $D$  as $n\to\infty$. If $v\in D_n\cap V_j$, then by \eqref{eq:crucial step} and the last inequality we get
		\begin{align*}
			(1+\varepsilon)^{-3}\le \frac{M_D(x,v)}{M_{D}(x,z_j)}\cdot\frac{M_D(x,z_j)}{M_{D_n}(x,z_j)}\cdot\frac{M_{D_n}(x,z_j)}{M_{D_n}(x,v)}\le (1+\varepsilon)^3.
		\end{align*}
		Therefore
		\begin{align}\phantomsection\label{M_D ocjena}
			(1+\varepsilon)^{-3}\le \frac{\int_{D\cap V}M_D(x,y)\eta_n(dy)}{\int_{D\cap V}M_{D_n}(x,y)\eta_n(dy)}\le (1+\varepsilon)^3,\quad n>k.
		\end{align}
		Notice that $(\eta_n)_n$ also weakly converges to $\mu^*$ on $D^*\cap V$ and that $x,x_0\notin D\cap V$. Recall that $M_D(x,\cdot)$ is continuous and bounded on $D^*\cap V$ (see Lemma \ref{lemma 12}). Therefore 
		\begin{align*}
			\int_{D\cap V}M_D(x,y)\eta_n(dy)\to \int_{D^*\cap V}M_D(x,y)\mu^*(dy)=\int_{\partial^*D}M_D(x,y)\mu^*(dy).
		\end{align*}
		Also, note that $f(x)=\int_{D\cap V}M_{D_n}(x,y)\eta_n(dy)+\int_{D\cap V^c}M_{D_n}(x,y)\eta_n(dy)$ and that there is $k$ so large  such  that $D\cap V^c\subset D_k$. 
		Hence
		\begin{align*}
			&\int_{D\cap V^c}M_{D_n}(x,y)\eta_n(dy)\le \int_{D_k}M_{D_n}(x,y)\eta_n(dy)\\
			&\qquad\qquad\qquad=\int\limits_{D_k}G_{D_n}(x,v)\int\limits_{D\setminus D_n}j(|v-y|)f(y)dydv\\
			&\qquad\qquad\qquad\le c_k\left(\int\limits_{D_k}G_D(x,v)dv\right)\left(\int\limits_{D\setminus D_n}f(y)(1\wedge j(|y|))dy\right)\overset{n\to\infty}{\longrightarrow} 0,
		\end{align*}
		since $f\in\mathcal{L}^1$ by Lemma \ref{l:harmonijske su u L^1} and since $G_D(x,\cdot)\in L^1_{loc}$ which we get from \eqref{eq:G comparability}.
		By letting $n\to\infty$ in \eqref{M_D ocjena} we obtain
		\begin{align*}
			(1+\varepsilon)^{-3}\le \frac{\int_{\partial^*D}M_D(x,y)\mu^*(dy)}{f(x)}\le (1+\varepsilon)^3.
		\end{align*}
		i.e. $f(x)=\int_{\partial^* D}M_D(x,y)\mu^*(dy)$.
		
		We now prove that the measure $\mu^*$ is concentrated on $\partial_MD$. Let $x\in U\subsub D$. If $y \in \partial^*D$, then by Theorem \ref{Martinova jezgra} $M_D(x, y)\ge\int_{D\setminus U} M_D(z, y)\omega_U^x(dz)$ and equality holds if and only if $y\in\partial_MD$. By Fubini’s theorem
		\begin{align*}
			0=f(x)-\int\limits_{D\setminus U}f(z)\omega_U^x(dz)=\int\limits_{\partial^* D}\left(M_D(x,y)-\int_{D\setminus U} M_D(z, y)\omega_U^x(dz)\right)\mu^*(dy),
		\end{align*}
		hence $\mu^*(\partial^* D\setminus \partial_MD)=0$.
		
		Now we prove uniqueness. Consider first the case $f(\,\cdot\,)=M_D(\,\cdot\,,z_0)=M_D\delta_{z_0}(\,\cdot\,)$ and suppose that there is another measure $\mu$ on $\partial_MD$ such that $f=M_D\mu$. If $z_0$ is finite, then the uniqueness is proved in the same way as in \cite{bogdan_est_and_struct}. Therefore, we deal with the case $z_0=\infty$.
		For $s>0$ define $D_s=D\cap B_s$ and take $R>0$ such that \eqref{(3.4)} is true, i.e.  $M_D(x,\infty)=\mathbb{E}_x[M_D(X_{\tau_{D_R}},\infty)]$, $x\in D_R$. Define  the  function $g:\R^d\to[0,\infty)$ as $g(x)=\int_{|y|<R}M_D(x,y)\mu(dy)$. For $x\in D\setminus D_{2R}$, by Fubini's theorem and the comment about \eqref{(3.14)} in Remark \ref{o mvp za martinove}, we have that
		\begin{align*}
			\int_{D_{2R}} g(z)\omega_{D\setminus D_{2R}}^x(dz)&=\int_{|y|<R}\left(\int_{D_{2R}}M_D(z,y)\omega_{D\setminus D_{2R}}^x(dz)\right)\mu(dy)\\
			&=\int_{|y|<R}M_D(x,y)\mu(dy)=g(x).
		\end{align*}
		Also, for $x\in D_{3R}$ we have $g(x)=\int_{D\setminus D_{3R}} g(z)\omega_{D_{3R}}^x(dz)$. Indeed, $g\le f$ and $f(x)=\int_{D\setminus D_{3R}} f(z)\omega_{D_{3R}}^x(dz)$ because of \eqref{(3.4)} so Lemma \ref{lemma 9} yields the claim.	Lemma \ref{lemma 10} yields $g=0$ on whole $D$, in particular $g(x_0)=\mu(\{|y|<R\})=0$. Since this is true for all big $R>0$, we see that $\mu$ is concentrated at the point at infinity. Thus, we have uniqueness for the function $f(\,\cdot\,)=M_D(\cdot,\infty)$.
		
		Consider now $f=M_D\mu$ for a finite measure $\mu$  on $\partial_MD$  and let $(\eta_{D_n}f)_n$ be the corresponding sequence of measures for $f$ from the beginning of the proof. We want to show that $\mu^*=\mu$. Since $(\eta_{D_n}f)_n$ converges weakly to $\mu^*$, by uniqueness of the weak limit it is enough to show that for every relatively open set $A\subset \overline{D}$ we have $\liminf_n\eta_{D_n}f(A)\ge \mu(A)$. To this end, using Fubini's theorem, Fatou's lemma, and what was already proven for the case of  the  Dirac measures we have
		\begin{align*}
			&\liminf_{n\to\infty}\eta_{D_n}f(A)=\liminf_{n\to\infty}	\int\limits_{A}G_{D_n}(x_0,v)\left(\int\limits_{D\setminus D_n}j(v,y)M_D\mu(y)dy\right)dv\\
			&=\liminf_{n\to\infty}\int\limits_{\partial_MD}\left(\int\limits_{A}G_{D_n}(x_0,v)\left(\int\limits_{D\setminus D_n}j(|v-y|)M_D(y,z)dy\right)dv\right)\mu(dz)\\
			&\ge \int\limits_{\partial_MD}\liminf_{n\to\infty}\left(\int\limits_{A}G_{D_n}(x_0,v)\left(\int\limits_{D\setminus D_n}j(|v-y|)M_D(y,z)dy\right)dv\right)\mu(dz)\\
			&=\int\limits_{\partial_MD}\liminf_{n\to\infty}\eta_{D_n}\big(M_D(\cdot,z)\big)
			(A)\mu(dz)\ge \int\limits_{\partial_MD}\delta_z(A)\mu(dz)=\mu(A).
		\end{align*}
		Thus, we have proved uniqueness.
		
		Notice that due to uniqueness of the measure $\mu$, any choice of the sequence $(D_n)_n$ from the beginning of the proof gives $\mu$ as the limit of $\eta_{D_n}f$ so we have proved that $W_Df$ is well defined and that $W_Df=\mu$. 
	\end{proof}
	
	\begin{rem}
		Since for a finite measure $\mu$ on $\partial_MD$ we have that $M_D\mu$ is $L$-harmonic with zero outer charge, we have that $M_D\mu\in C^\infty(D)\cap \mathcal{L}^1$ and if $D$ is bounded we have $M_D\mu\in L^1(D)$, see Lemma \ref{l:harmonijske su u L^1}, and  Theorem  \ref{r:neprekidnost harmoncijskih Cinfty}.
	\end{rem}
	
	Combining Propositions \ref{lemma 1.17_bodi et ali} and \ref{lemma 14}, we get that (under the additional assumptions \ref{H2} and \ref{E2} if $D$ is unbounded)
	\begin{align}\phantomsection\label{ne znam kako bih ovo nazvao}
		W_D[G_Df+P_D\lambda+M_D\mu]=\mu.
	\end{align}
	
	\begin{cor}\phantomsection\label{c:beskonacnost od M_Dmu}
		Let $D$ be an open set. If $D$ is unbounded suppose \ref{H2} and \ref{E2}. Let $\mu$ be a measure on $\partial_MD$. $M_D\mu(x)=\infty$ for some $x\in D$ if and only if $M_D\mu\equiv\infty$ in $D$, and in that case $\mu$ is an infinite measure.
	\end{cor}
	\begin{proof}
		Lemma \ref{lemma 12} yields that $\partial^*D \ni z\mapsto M_D(x,z)$ is bounded from below and above for every $x\in D$. Now the claim easily follows.
	\end{proof}
	
	\begin{thm}[Representation of non-negative $L$-harmonic functions]\phantomsection\label{t:reprezentacija_L_harm}
		Let $D$ be an open set. If $D$ is unbounded  additionally assume \ref{H2} and \ref{E2}. If $f$ is a non-negative function, $L$-harmonic in $D$ with  a  non-negative outer charge $\lambda$, then
		there is a unique finite measure $\mu_f$ on $\partial_MD$ such that $f = P_D\lambda + M_D\mu_f$ on $D$.
	\end{thm}
	\begin{proof}
		The proof is the same as in \cite[Lemma 13]{bogdan_est_and_struct}.
	\end{proof}
	
	\subsection*{Acknowledgement}
	
	This research was supported in part by the Croatian Science Foundation under the project 4197. The author would like to thank Prof. Zoran Vondraček and Prof. Vanja Wagner for many discussions on the topic and for
	helpful comments on the presentation of the results.  The author would also like to thank the Referees for helpful comments that improved the presentation and the results. 
	

	\bibliographystyle{abbrv}
	\bibliography{biblio1}
	
	\bigskip
	
	\noindent{\bf Ivan Bio\v{c}i\'c}
	
	\noindent Department of Mathematics, Faculty of Science, University of Zagreb, Zagreb, Croatia,
	
	\noindent Email: \texttt{ibiocic@math.hr}

\end{document}